\documentclass[11pt]{article}
\usepackage{amsfonts}
\usepackage{amsmath}
\usepackage{amssymb}

\newcommand{\Z}{\mathbb{Z}}
\newcommand{\Q}{\mathbb{Q}}
\newcommand{\R}{\mathbb{R}}
\newcommand{\C}{\mathbb{C}}

\usepackage{physics}
\usepackage{fancyhdr}
\usepackage{hyperref}
\usepackage{color}
\usepackage{enumerate}

\usepackage{cite}
\usepackage{amsthm}

\newtheorem{theorem}{Theorem}

\newtheorem{proposition}[theorem]{Proposition}
\newtheorem{corollary}[theorem]{Corollary}
\newtheorem{lemma}[theorem]{Lemma}
\theoremstyle{remark}
\newtheorem{remark}[theorem]{Remark}

\theoremstyle{definition}
\newtheorem{definition}[theorem]{Definition}
\newtheorem*{definition*}{Definition}
\newtheorem*{claim*}{Claim}
\newtheorem*{theorem*}{Theorem}
\newtheorem*{corollary*}{Corollary}

\title{Rankin--Cohen Type Differential Operators \newline on Hermitian Modular Forms}
\author{Francis Dunn\footnote{Affiliation: University of Oregon, Department of Mathematics, USA. \newline\hspace*{1.8em} Email: fdunn@uoregon.edu\newline\hspace*{1.8em} ORCID: 0000-0003-1723-1190 }}
\date{}

\begin{document}
	
	\maketitle
	\vfill
	\begin{center}\subsubsection*{Abstract}\end{center}
	We construct Rankin--Cohen type differential operators on Hermitian modular forms of signature $(n,n)$. The bilinear differential operators given here specialize to the original Rankin--Cohen operators in the case $n=1$, and more generally satisfy some analogous properties, including uniqueness.
	Our approach builds on previous work by Eholzer--Ibukiyama in the case of Siegel modular forms, together with results of Kashiwara--Vergne on the representation theory of unitary groups.
	\vfill 
	
	\tableofcontents
	
	\newpage
	\section{Introduction}

	\subsection{Overview} 
	
	In the classical setting, recall that the derivative of a holomorphic modular form of integral weight on the complex upper half-plane is not in general a modular form since the derivative fails to satisfy the correct transformation properties. However, in 1956 R. A. Rankin was able to describe differential operators sending modular forms to modular forms \cite{Rankin}. H. Cohen exhibited a special case in \cite{Cohen}, introducing particular bilinear differential operators on the graded ring of modular forms. These Rankin--Cohen operators or Rankin--Cohen brackets have proven to be interesting objects to study (see, for example, \cite{Zag}) and provide the unique way of combining derivatives of two modular forms to produce another modular form of a higher weight. That is, for modular forms $f$ and $g$ of weights $k$ and $\ell$ respectively, for each integer $v>0$ the Rankin--Cohen bracket indexed by $k$, $\ell$ and $v$ is the unique bilinear differential operator (i.e. a bilinear combination of the derivatives of $f$ and $g$), up to rescaling, which gives a modular form of weight $k+\ell+2v$. Furthermore, for $v>0$ the resulting modular forms are cusp forms, which hold special significance in the literature. In this paper we tackle the case of scalar-valued Hermitian modular forms of signature $(n,n)$, and produce analogous results in this setting.

	\subsection{Main Results}

	We prove the following results for Hermitian modular forms of signature $(n,n)$ under some technical conditions on the weights $k_{1},\dots,k_{r}$ and $v$ below:
	
	\begin{theorem*}[A] (Imprecise version of Theorem \ref{thm} and Corollary \ref{coro}.) \label{thm a} Covariant multilinear differential operators sending $r$ scalar-valued Hermitian modular forms of integer weights $k_{1},\dots,k_{r}$ respectively to a scalar-valued Hermitian modular form of weight $v+\sum_{i=1}^{r}k_{i}$ for a fixed $v\geq 0$ correspond to a certain space of homogeneous pluri-harmonic polynomials (as defined in \S\ref{section ph poly}). \end{theorem*}

	\begin{theorem*}[B] (Imprecise version of Theorem \ref{thm coeff}.) \label{thm b} In the case $r=2$, for each $v>0$ and $n>1$ there is a unique (up to rescaling) differential operator sending scalar-valued Hermitian modular forms of weights $k_{1}$ and $k_{2}$ respectively to a Hermitian modular form of weight $k_{1}+k_{2}+2v$. We also give integral linear relations satisfied by the coefficients of the differential operators. \end{theorem*}
	
	As a consequence of Theorem B, we extract that for $v>0$ the image of these bilinear differential operators is contained in the space of cusp forms in Proposition \ref{prop cusp}. In addition, since the relations between coefficients of the operators are integral, the ring of definition for the modular forms is preserved.\\
	
	This paper builds on several previous works generalizing Rankin--Cohen type operators beyond the classical setting. Choie--Eholzer used Jacobi forms \cite{ChoEho} to constrct Rankin--Cohen type operators for Siegel modular forms, which Eholzer--Ibukiyama \cite{EhoIbu97} and Ibukiyama \cite{Ibu99} later generalized using a different approach involving theory of pluriharmonic polynomials, in part developed by Kashiwara--Vergne \cite{KasVer}. In the case of Hermitian modular forms, Martin--Senadheera \cite{MarSen} used Jacobi forms to produce analogues of Rankin--Cohen brackets for Hermitian modular forms over the field extension $\Q(i)/\Q$ and signature $(2,2)$. This paper follows a similar approach to \cite{EhoIbu97}, again using results of \cite{KasVer}, to produce Rankin--Cohen type operators on Hermitian modular forms in a more general setting than before, where we allow the underlying field to be any CM field and for signature $(n,n)$. 
	
	\subsection{Connection to other Results}
	
	Rankin--Cohen type differential operators have been valuable tools to study and prove results concerning automorphic forms. In the classical setting, the Rankin--Cohen brackets endow rich algebraic structure on the graded algebra of modular forms as described by Zagier \cite{Zag}. Their algebraic properties have been the subject of further study, for example by El Gradechi \cite{Gradechi}. Rankin--Cohen operators have also found use in settings such as the study of mod $p$ differential operators \cite{EFGMM}. This paper helps lay the groundwork for similar applications in the setting of Hermitian modular forms. 
 
    After posting the first version of this paper, the author became aware of the work of Ban \cite{Ban06} studying Lie-theoretic operators corresponding to Rankin--Cohen operators for automorphic forms on $\text{SU}(p,q)$. While the setting is similar, the perspective and approach of Ban is very different to that of this paper, and there is not an explicit construction of the operators concerned. It would be interesting to study how these results correspond. 
    
    Furthermore, the image of the Rankin--Cohen operators in the elliptic and Siegel settings has been investigated in, for example, \cite{Kohnen1991} and \cite{JhaSahu}. Similar investigations could be now be carried out for the operators constructed in this paper, however some technical results would be necessary for the same approach to be used in the Hermitian setting.  
	
	\subsection{Organization}
	
	In \S 2 we recall the relevant definitions for Hermitian modular forms  and introduce the particular spaces of polynomials of interest to us. In \S 3 we establish the correspondence between differential operators on Hermitian modular forms and certain polynomials, stating and proving the precise version of Theorem A. In \S 4 we prove Theorem B by giving an explicit description of the particular space of polynomials that we are interested in, and then finding linear relations that coefficients of the specific polynomials we care about must satisfy. Finally, in \S 5 we consider the effect of our differential operators on Fourier expansions of Hermitian modular forms, and prove that the non-trivial bilinear operators produce cusp forms.
	
	\subsection{Acknowledgements}
	
	The author thanks their supervisor Ellen Eischen, as well as Ben Elias, Ben Young, Jon Brundan, and Nicolas Addington for useful conversations relating to the work in this paper. The author is also grateful to Tomoyoshi Ibukiyama for his comments on an early version of this paper, as well as the referees for their helpful suggestions. This work was supported in part by the Jack and Peggy Borsting Award.

	\section{Preliminaries}

	\subsection{Unitary Groups and Hermitian Modular Forms}

	First, we recall the key definitions for the specific kind of modular forms with which we are concerned. 
	
	\begin{definition}[Hermitian Upper Half-space] \label{def: Hn}
		We define the \textit{Hermitian upper half-space} as
		\[\mathbb{H}_{n}:=\left\{Z\in M_{n,n}(\C) \bigm| \frac{1}{2i}(Z-Z^{\ast})>0 \right\} \]
		where $Z^{\ast}$ denotes the conjugate-transpose of $Z$, and we write $M>0$ for a square matrix $M$ to mean that $M$ is positive definite.
	\end{definition}
	
	\vspace*{1ex}
	
	\begin{definition}[Unitary Group $U(n,n)$)] \label{def: Unn}
		Let $J=\begin{pmatrix} 0 & -i1_{n}\\ i1_{n} & 0\end{pmatrix}\in M_{2n,2n}$. Then we define
		\[ U(n,n):=\{ \gamma\in M_{2n,2n}(\C)\mid \gamma^{\ast}J\gamma=J\}\]
		to be the \textit{unitary group of signature $(n,n)$}. A collection of generators of $U(n,n)$ is given by matrices of the following three forms: 
		\begin{enumerate}[(I)]
			\item $\begin{pmatrix} A & 0\\ 0 & (A^{\ast})^{-1}\end{pmatrix}$ for any  $A\in\text{GL}(n,\C)$
			\item $\begin{pmatrix} 1_{n} & B\\ 0 & 1_{n}\end{pmatrix}$ where $B$ is any $n$-by-$n$ Hermitian matrix
			\item $\begin{pmatrix} 0 & -1_{n}\\ 1_{n} & 0\end{pmatrix}$.
		\end{enumerate}
	\end{definition}
	
	\vspace*{1ex}

	We have an action of $U(n,n)$ on $\mathbb{H}_{n}$ given by 
	\[ \begin{pmatrix} A & B\\ C & D\end{pmatrix}\cdot Z= (AZ+B)(CZ+D)^{-1} \]
	for $\begin{pmatrix} A & B\\ C & D\end{pmatrix}\in U(n,n)$ and $Z\in\mathbb{H}_{n}$.
	Explicitly, for the three types of generators above we have
	\begin{enumerate}[(I)]
		\item $\begin{pmatrix} A & 0\\ 0 & (A^{\ast})^{-1}\end{pmatrix} \cdot Z = AZA^{\ast}$
		\item $\begin{pmatrix} 1_{n} & B\\ 0 & 1_{n}\end{pmatrix} \cdot Z = Z+B$
		\item $\begin{pmatrix} 0 & -1_{n}\\ 1_{n} & 0\end{pmatrix}\cdot Z = -Z^{-1}$.
	\end{enumerate}
	
	\vspace*{1ex}
	
	\begin{definition}[Hermitian Modular Forms] \label{def: hmf}
		Fix a quadratic imaginary extension $L$ of a totally real number field $L^{+}$. Let $\Gamma$ be a congruence subgroup of $U(n,n)(\mathcal{O}_{L})$.
		A (scalar-valued) Hermitian modular form of level $\Gamma$ and weight $k\in\Z$ is a holomorphic function $F:\mathbb{H}_{n}\to\C$ such that for all $\gamma=\begin{pmatrix}A&B\\C&D\end{pmatrix}\in\Gamma$
		\[ F(Z)=\text{det}(CZ+D)^{-k}F(\gamma Z). \]
		
	\end{definition}
	
	\vspace*{1ex}
	
	\subsection{Pluriharmonic Polynomials} \label{section ph poly}
	
	The differential operators we shall construct are described in terms of specific polynomials, which we shall now define (as in \cite{KasVer}).
	
	\vspace*{1ex}
	
	\begin{definition}[Harmonic and Pluriharmonic]\label{def: harmonic}
		
		Let $P(X,Y)$ be a polynomial in matrix variables $X=(x_{s,u})$ and $Y=(y_{s,u})$ in $M_{n,k}(\C)$.  Set
		
		\[ \Delta_{s,t}(X,Y)=\sum_{u=1}^{k}\frac{\partial^{2}}{\partial x_{s,u}\partial y_{t,u}} \hspace*{5ex} (1\leq s,t\leq n). \]

		Then we say that $P$ is \textit{harmonic} if $\sum_{s=1}^{n}\Delta_{s,s}(X,Y)P=0$ and \textit{pluriharmonic} if $\Delta_{s,t}(X,Y)P=0$ for every $1\leq s,t\leq n$.  \end{definition}
	
	\vspace*{1ex}
	
	\begin{definition}[Homogeneous] \label{def: homogeneous} There is a group action of $\text{GL}(n,\C)\times\text{GL}(n,\C)\times\text{GL}(k,\C)$ on polynomials $P(X,Y)$ given by 
		
		\[ (A,B,C)\cdot P(X,Y)=P(AXC^{-1},BY{}^{t}C), \]

		\noindent where we write ${}^{t}M$ to denote the transpose of a matrix $M$. We say that $P$ is \textit{homogeneous of degree $v$} if, setting $C=1_{k}$, for every $(A,B)\in\text{GL}(n,\C)\times\text{GL}(n,\C)$ we have
		
		\[ (A,B,1_{k})\cdot P(X,Y)=\text{det}(A)^{v}\text{det}(B)^{v}\times P(X,Y).\]
		
	\end{definition}
	
	\vspace*{1ex}
	
	\begin{remark} \label{rem: harmonic}
		Note that $P(X,Y)$ is pluriharmonic if and only if $P(AX,BY)$ is harmonic for every $(A,B)\in\text{GL}(n,\C)\times\text{GL}(n,\C)$. So a homogeneous polynomial is pluriharmonic if and only if it is harmonic. 
	\end{remark}
	
	\vspace*{1ex}
	
	\begin{definition}[Invariance and Association] \label{def: invariance}
		For positive integers $r$ and $k_{1},\dots,k_{r}$ with $k_{1}+\dots+k_{r}=k$, let $K=K(k_{1},\dots,k_{r})$ be the (image of the) block-diagonal embedding 
		\[ \text{GL}(k_{1},\C)\times\dots\times\text{GL}(k_{r},\C)\hookrightarrow\text{GL}(k,\C). \]
		\noindent We say $P(X,Y)$ is \textit{$K$-invariant} if $P(XC^{-1},Y{}^{t}C)=P(X,Y)$ for every matrix $C$ in $K$. For any such polynomial $P$, there is a polynomial $Q(R_{1},\dots,R_{r})$ with variables $R_{1},\dots,R_{r}\in M_{n,n}(\C)$ such that, writing $X=(X_{1},\dots,X_{r})$ and $Y=(Y_{1},\dots,Y_{r})$ with $X_{i},Y_{i}\in M_{n,k_{i}}$ for $i=1,\dots,r$, $P(X,Y)=Q(X_{1}{}^{t}Y_{1},\dots,X_{r}{}^{t}Y_{r})$. For such a $P$ and $Q$ we say that $Q$ is \textit{associated to} $P$. \end{definition}
	
	\vspace*{1ex}
	
	The particular spaces of polynomials we are concerned with are as follows:
	
	\begin{definition}[Notation] \label{def: notation}
		Denote by $\mathcal{P}_{n,v}(k_{1},\dots,k_{r})$ the space of all $K$-invariant homogeneous polynomials $P$ of degree $v$.
		
		Let $\mathcal{Q}_{n,v}(r)$ be the space of all polynomials $Q$ in the coefficients of $n\times n$ matrices $R_{1},\dots,R_{r}$ such that 
		\[Q(AR_{1}{}^{t}B,\dots, AR_{r}{}^{t}B)=\text{det}(A)^{v}\text{det}(B)^{v}Q(R_{1},\dots,R_{r}) \] for all $(A,B)\in\text{GL}(n,\C)\times\text{GL}(n,\C)$. Put $\mathcal{Q}_{n}(r):=\oplus_{v\geq 0}\mathcal{Q}_{n,v}(r)$.  
		
		Let $\mathcal{H}_{n,v}(k_{1},\dots,k_{r})$ be the subspace of $\mathcal{Q}_{n,v}(r)$ consisting of the elements $Q\in\mathcal{Q}_{n,v}(r)$ such that $Q(X_{1}{}^{t}Y_{1},\dots,X_{r}{}^{t}Y_{r})$ is pluriharmonic considered as a polynomial in the entries of $X=(X_{1},\dots,X_{r})$ and $Y=(Y_{1},\dots,Y_{r})$, $X_{i},Y_{i}\in M_{n,k_{i}}$. \end{definition}
	
	\vspace*{1ex}
	
	\begin{remark} 
		When $k_{i}\geq n$ for each $1\leq i\leq r$, $\mathcal{H}_{n,v}(k_{1},\dots,k_{r})$ is the space of polynomials $Q$ associated to the degree-$v$ homogeneous $K$-invariant pluriharmonic polynomials $P\in\mathcal{P}_{n,v}(k_{1},\dots,k_{r})$ as above.
	\end{remark}

	\vspace*{2ex}

	We shall be concerned with differential operators associated with these pluriharmonic polynomials, with the association made explicit by the following definition:

	\begin{definition}[Differential Operators]
		
		For a matrix variable $Z=(z_{s,t})\in M_{n,n}$ we define the $n\times n$ matrix of derivatives $\frac{\partial}{\partial Z}$ by
		
		\[ \left(\frac{\partial}{\partial Z}\right)_{s,t} = \frac{\partial}{\partial z_{t,s}}. \]
		
		Then, given a polynomial $Q(R_{1},\dots,R_{r})$ in the entries of $n\times n$ matrices $R_{1},\dots,R_{r}$, we can define a differential operator on holomorphic functions $M_{n,n}^{r}\to\C$ in variables $Z_{1},\dots,Z_{r}$ by $Q\left(\frac{\partial}{\partial Z_{1}},\dots,\frac{\partial}{\partial Z_{r}}\right)$. \end{definition}
	
	\vspace*{1ex}
	
	\section{Characterization of Differential Operators}
	
	In this section we give the full, precise statement of Theorem \nameref{thm a} and prove the characterization of differential operators on Hermitian modular forms in terms of the pluriharmonic polynomials defined above. 
	
	\begin{theorem}\label{thm} Let $L/L^{+}$ be a quadratic imaginary extension of a totally real number field, $k_{1},\dots,k_{r},v\in\Z_{>0}$ with each $k_{i}\geq n$ and $\frac{\abs{\mu(L)}}{2}[\mathcal{O}_{L}^{\times}:\mu(L)\mathcal{O}^{\times}_{L^{+}}]$ dividing $v$, where $\mu(L)$ is the set of roots of unity contained in $L$. For a polynomial $Q$ on $M_{n,n}^{r}$, we have
		
		\begin{multline} \label{eq comm rel}
			Q\left(\frac{\partial}{\partial Z}\right)\left(F(\gamma Z_{1},\dots,\gamma Z_{r})\prod_{i=1}^{r}\text{\emph{det}}(CZ_{i}+D)^{-k_{i}}\right)|_{Z_{1}=\dots=Z_{r}=Z} \\ =\text{\emph{det}}(CZ+D)^{-(k+2v)}Q\left(\frac{\partial}{\partial \gamma Z}\right)\left(F(\gamma Z,\dots,\gamma Z)\right)
		\end{multline}
		
		\noindent for $k=\sum_{j=1}^{r}k_{r}$ and every $\gamma=\begin{pmatrix} A & B\\ C & D\end{pmatrix}\in  U(n,n)(\mathcal{O}_{L})$, if and only if $Q\in\mathcal{H}_{n,v}(k_{1},\dots,k_{r})$.\end{theorem}
	
	\begin{remark}
		It was remarked to the author by T. Ibukiyama that the condition on the degree $v$ is essentially the condition of a real Lie group for the unitary group in question.
	\end{remark}
	
	\begin{corollary}[Differential Operators on Hermitian Modular Forms] \label{coro}
		If, in the setting of the theorem, we have $F(Z_{1},\dots,Z_{r})=F_{1}(Z_{1})\dots F_{r}(Z_{r})$ with each $F_{j}$ ($j=1,\dots,r$) a scalar-valued Hermitian modular form of weight $k_{i}$ of some level $\Gamma$ then  applying the differential operator associated to some polynomial $Q$ gives a Hermitian modular form of weight $2v+k$ and level $\Gamma$ if  $Q\in\mathcal{H}_{n,v}(k_{1},\dots,k_{r})$. \end{corollary}

	\begin{proof}[Proof of Corollary \ref{coro}]
		This follows because, for such an $F$,
		\begin{equation*}
			\begin{split}
				& \left( Q \left( \frac{\partial}{\partial Z} \right) (F_{1}\dots F_{r})\right)(Z_{1},\dots,Z_{r})|_{Z_{1}=\dots=Z_{r}=Z}\\
				& \hspace*{3ex} =Q\left(\frac{\partial}{\partial Z}\right)\left(F_{1}(Z_{1})\dots F_{r}(Z_{r})\right)|_{Z_{1}=\dots=Z_{r}=Z} \\
				& \hspace*{3ex} =Q\left(\frac{\partial}{\partial Z}\right)\left(F_{1}(\gamma Z_{1})\dots F_{r}(\gamma Z_{r})\prod_{i=1}^{r}\text{det}(CZ_{i}+D)^{-k_{i}}\right)|_{Z_{1}=\dots=Z_{r}=Z} \\
				& \hspace*{3ex} = \text{det}(CZ+D)^{-(k+2v)}Q\left(\frac{\partial}{\partial \gamma Z}\right)\left(F_{1}(\gamma Z)\dots F_{r}(\gamma Z)\right) \\
				& \hspace*{3ex} = \text{det}(CZ+D)^{-(k+2v)}\left(Q\left(\frac{\partial}{\partial Z}\right)(F_{1}\dots F_{r})\right)(\gamma Z,\dots,\gamma Z),
			\end{split}
		\end{equation*}	
		where the penultimate equality is the relation (\ref{eq comm rel}). \end{proof}

	\begin{remark}
		The reverse implication for Corollary \ref{coro} is not quite true. One counterexample, pointed out by T. Ibukiyama, is that if we choose $k_{1}$ and $\Gamma$ such that the only Hermitian modular form of weight $k_{1}$ and level $\Gamma$ is the zero function, then the differential operator corresponding to any polynomial $Q$ will act on the zero function and return the zero function, which is of course a Hermitian modular form of every weight and level.
		
	\end{remark}
	
	\begin{proof}[Proof of Theorem \ref{thm}] We start by supposing $Q\in\mathcal{H}_{n,v}(k_{1},k_{2},\dots,k_{r})$. It is sufficient to show that the result holds for a collection of generators of $U(n,n)$ such as the one given in Definition \ref{def: Unn}.
		
		\noindent\textbf{Case I} We show the relation for $\gamma=\begin{pmatrix} A & 0\\ 0 & (A^{\ast})^{-1}\end{pmatrix}$ for $A\in\text{GL}(n,\C)$. First, note that 
		\[ \frac{\partial}{\partial Z}= A^{\ast}\frac{\partial}{\partial W}A, \]
		where $W=\gamma\cdot Z=AZA^{\ast}$. Furthermore, note that $A\in\text{GL}(n,\mathcal{O}_{L})$ and $\text{det}(A)\in\mathcal{O}_{L}^{\times}$. For $v$ as in the statement of the theorem, $\text{det}(A)^{v}$ and $\text{det}(A^{\ast})^{v}$ must lie inside $\mathcal{O}^{\times}_{L^{+}}\subset \R$ and thus 
		\[ \text{det}(A)^{v}\text{det}(A^{\ast})^{v}=\text{det}(A)^{2v}=\text{det}(A^{\ast})^{2v}.\]

		Then
		
		\begin{equation*}
			\begin{split}
				& Q\left(\frac{\partial}{\partial Z}\right)\left(F(AZ_{1}A^{\ast},\dots,A Z_{r}A^{\ast})\text{det}(A^{\ast})^{k}\right)|_{Z_{1}=\dots=Z_{r}=Z} \\
				& =\text{det}(A^{\ast})^{k}Q\left(\frac{\partial}{\partial Z}\right)\left(F(AZA^{\ast},\dots,A ZA^{\ast})\right) \hspace*{4ex} \text{\small{since $A$ and $A^{\ast}$ do not depend on $Z$}}\\
				& = \text{det}(A^{\ast})^{k}Q\left(A^{\ast}\frac{\partial}{\partial W}A\right)\left(F(AZA^{\ast},\dots,A ZA^{\ast})\right) \hspace*{4ex} \text{\small{by the equality above}}\\
				& =\text{det}(A^{\ast})^{k+v}\text{det}(A)^{v}Q\left(\frac{\partial}{\partial W}\right)\left(F(W,\dots,W)\right) \hspace*{4ex} \text{\small{since $Q$ is homogeneous of degree $v$}} \\
				& =\text{det}(A^{\ast})^{k+2v}Q\left(\frac{\partial}{\partial W}\right)\left(F(W,\dots,W)\right) \hspace*{4ex} \text{\small{by the observation above,}}
			\end{split}
		\end{equation*}
		
		and this is precisely the relation (\ref{eq comm rel}) in this case. \\
		
		\noindent\textbf{Case II} We consider $\gamma=\begin{pmatrix} 1_{n} & B\\ 0 & 1_{n}\end{pmatrix}$ with $B$ Hermitian, $B=B^{\ast}$. Then $W=\gamma\cdot Z=Z+B$ and $\frac{\partial}{\partial Z}=\frac{\partial}{\partial W}$.
		
		We have
		\begin{equation*}
			\begin{split}
				& Q\left(\frac{\partial}{\partial Z}\right)\left(F(Z_{1}+B,\dots, Z_{r}+B)\text{det}(1_{n})^{-k}\right)|_{Z_{1}=\dots=Z_{r}=Z} \\
				& = Q\left(\frac{\partial}{\partial Z}\right)\left(F(Z+B,\dots, Z+B)\right)\\
				& = Q\left(\frac{\partial}{\partial W}\right)\left(F(W,\dots, W)\right),
			\end{split}
		\end{equation*}
		and we are done.
		
		\noindent\textbf{Case III} We take $\gamma=\begin{pmatrix} 0 & -1_{n}\\ 1_{n} & 0\end{pmatrix}$. So $W=\gamma\cdot Z=-Z^{-1}$ and
		
		\[ \frac{\partial}{\partial Z}= (-Z^{-1})\frac{\partial}{\partial W}(-Z^{-1}). \]
		
		This case requires a little more work. We argue in a similar fashion to the analagous Theorem 2 of Ibukiyama \cite{Ibu99}. We see that there exist functions $G_{M}(Z)$, independent of $F$, such that
		\[ Q\left(\frac{\partial}{\partial Z}\right)(F(W,\dots,W)\text{det}(-Z)^{-k})=\sum_{\abs{M}\leq\text{deg}(Q)}G_{M}(Z)\left(Q_{M}\left(\frac{\partial}{\partial Z}\right) F\right)(W), \]
		where the sum runs over $n\times n$ matrices $M=(m_{s,t})$ with non-negative integer entries satisfying 
		\[\abs{M}:=\sum_{1\leq s,t,\leq n}m_{s,t}\leq \text{deg}(Q),\] and $Q_{M}$ is the monomial associated to $M$, i.e.  \[ Q_{M}\left(\frac{\partial}{\partial Z}\right) = \prod_{1\leq s,t\leq n}\left(\frac{\partial}{\partial Z}\right)_{s,t}^{m_{s,t}} = \prod_{1\leq s,t\leq n}\frac{\partial^{m_{s,t}}}{\partial z_{t,s}^{m_{s,t}}}. \]
		
		Therefore, if we can verify that the commutation relation holds for a particular holomorphic function $F$ for which $Q_{M}\left(\frac{\partial}{\partial Z}\right)(F)$ are linearly independent at any point $Z$ (or linearly independent over the ring of holomorphic functions), then the commutation relation holds for all holomorphic functions.
		
		We take $F_{0}(Z):=\exp\left(i\Tr (Y^{\ast}ZY)\right)$ where $Y$ is an $n\times n$ matrix of independent variables $y_{s,t}$. Then
		
		\[ Q_{M}\left(\frac{\partial}{\partial Z}\right)(F_{0})(Z)= i^{\abs{M}}Q_{M}(YY^{\ast})F_{0}(Z).\]
		The $y_{s,t}$ are independent variables, and each $k_{i}\geq n$ so the components of $YY^{\ast}$ are algebraically independent. So, the values of the functions $Q_{M}\left(\frac{\partial}{\partial Z}\right)(F_{0})$ are linearly independent at any $Z$. So, we prove the commutation relation for $\gamma=\begin{pmatrix} 0 & -1_{n}\\ 1_{n} & 0\end{pmatrix}$ and $F_{0}$. To do so, we use the following technical lemma, a more general version of which was proved by Kashiwara--Vergne: \phantom{\qedhere}
	\end{proof}

	\begin{lemma}{4.2 p.39 \cite{KasVer}} \label{KV4.2}
		
		Let $P=P(X,Y)$ be a pluriharmonic polynomial. Then we have the following equality:
		\[ \int_{M_{n,k}(\C)} e^{i\Tr UY^{\ast}}e^{i\Tr U^{\ast}Y}e^{i\Tr U^{\ast}ZU}P(U,\overline{U})\abs{dU}^{2}=(-2\pi i)^{nk}\text{det}(-Z)^{-k}e^{i\Tr Y^{\ast}(-Z)^{-1}Y}P(Y,\overline{Y}),\] 
		where $\abs{dU}^{2}=2^{k}d(\text{Re}U_{1})d(\text{Im}U_{1})\dots d(\text{Re}U_{k})d(\text{Im}U_{k})$ with $U_{j}$ the $j$-th column of $U$. \end{lemma}
	
	\vspace*{3ex} 
	
	Now, applying this to the situation at hand:
	\begin{align*}
		& Q\left(\frac{\partial}{\partial Z}\right)(F_{0}(-Z^{-1})\text{det}(-Z)^{-k}) \\
		& =Q\left(\frac{\partial}{\partial Z}\right)\left( (-2\pi i)^{-nk}\int e^{i\Tr UY^{\ast}}e^{i\Tr U^{\ast}Y}e^{i\Tr U^{\ast}ZU}\abs{dU}^{2}\right) \hspace*{1em}\text{\small{by Lemma \ref{KV4.2} with $P$ trivial}}\\
		& = (-2\pi i)^{-nk}\int e^{i\Tr UY^{\ast}}e^{i\Tr U^{\ast}Y}Q\left(\frac{\partial}{\partial Z}\right)\left(e^{i\Tr U^{\ast}ZU}\right)\abs{dU}^{2} \hspace*{1em} \text{\small{by Leibniz integral rule}} \\
		& = (-2\pi i)^{-nk}\int e^{i\Tr UY^{\ast}}e^{i\Tr U^{\ast}Y}e^{i\Tr U^{\ast}ZU}i^{\text{deg}(Q)}Q(UU^{\ast})\abs{dU}^{2} \\
		& =(-2\pi i)^{-nk}i^{\text{deg}(Q)}\int e^{i\Tr UY^{\ast}}e^{i\Tr U^{\ast}Y}e^{i\Tr U^{\ast}ZU}P(U,\overline{U})\abs{dU}^{2} \hspace*{1em}\text{\small{where $P$ assoc. to $Q$}}\\
		& =(-2\pi i)^{-nk}i^{\text{deg}(Q)}(-2\pi i)^{nk}\text{det}(-Z)^{-k}e^{i\Tr Y^{\ast}(-Z)^{-1}Y}P(Y,\overline{Y}) \hspace*{1em}\text{\small{by Lemma \ref{KV4.2}}}\\
		& =\text{det}(-Z)^{-k}i^{\text{deg}(Q)}Q(YY^{\ast})e^{i\Tr Y^{\ast}(-Z)^{-1}Y} \\
		& =\text{det}(-Z)^{-k} Q\left(\frac{\partial}{\partial Z}\right)(F_{0}(-Z^{-1})) \\
		& =\text{det}(-Z)^{-k} Q\left((-Z^{-1})\frac{\partial}{\partial W}(-Z^{-1})\right)(F_{0}(-Z^{-1})) \\
		& =\text{det}(-Z)^{-(k+2v)} Q\left(\frac{\partial}{\partial W}\right)(F_{0}(W)) \hspace*{1em}\text{by the homogeneity of $Q$, and we are done.} 
	\end{align*}
	
	\vspace*{1em}
	
	So, if the polynomial $Q$ is a member of the space $\mathcal{H}_{n,v}(k_{1},\dots,k_{r})$ then equation (\ref{eq comm rel}) is satisfied for every element of our unitary group. Conversely, suppose that $Q$ is a polynomial on $M^{r}_{n,n}$ satisfying equation  (\ref{eq comm rel}), restated below:
	
	\begin{multline*} Q\left(\frac{\partial}{\partial Z}\right)\left(F(\gamma Z_{1},\dots,\gamma Z_{r})\prod_{i=1}^{r}\text{det}(CZ_{i}+D)^{-k_{i}}\right)|_{Z_{1}=\dots=Z_{r}=Z} \\ =\text{det}(CZ+D)^{-(k+2v)}Q\left(\frac{\partial}{\partial \gamma Z}\right)\left(F(\gamma Z,\dots,\gamma Z)\right) \hspace*{3ex} (\ref{eq comm rel}) \end{multline*} 
	
	for every matrix $\gamma=\begin{pmatrix} A&B\\C&D\end{pmatrix}\in U(n,n)$. We again consider $F_{0}(Z)=\text{exp}(i\Tr Y^{\ast}ZY)$. We start with the homogeneity property, and consider the relation (\ref{eq comm rel}) for the polynomial $Q$ in the case of a matrix $\gamma=\begin{pmatrix} A& 0 \\0 & (A^{\ast})^{-1}\end{pmatrix}$ for any $A\in\text{GL}(n,\mathcal{O}_{L})$:
	
	\[ Q\left(\frac{\partial}{\partial Z}\right)\left(\text{exp}(i\Tr Y^{\ast}AZA^{\ast}Y)\text{det}(A^{\ast})^{k}\right) = \text{det}(A^{\ast})^{k+2v}  Q\left(\frac{\partial}{\partial (AZA^{\ast})}\right)\left(\text{exp}(i\Tr Y^{\ast}AZA^{\ast}Y)\right). \]
	
	Computing each side, this becomes
	
	\[ \text{det}(A^{\ast})^{k}Q(A^{\ast}YY^{\ast}A)\text{exp}(i\Tr Y^{\ast}AZA^{\ast}Y) =\text{det}(A^{\ast})^{k+2v}Q(YY^{\ast})\text{exp}(i\Tr Y^{\ast}AZA^{\ast}Y).\]
	
	Since the matrix $Y$ is a matrix of independent variables and this holds for every value of $Z$ and every $A\in\text{GL}(n,\mathcal{O}_{L})$, we must have that
	
	\[ Q(A^{\ast}YY^{\ast}A)=\text{det}(A^{\ast})^{2v}Q(YY^{\ast}).  \]
	
	In particular, for $A\in\text{GL}(n,\mathcal{O}_{L})$ and $v$ as in the statement of Theorem \ref{thm}, 
	
	\[Q(A^{\ast}YY^{\ast}A)=\text{det}(A^{\ast})^{v}\text{det}(A)^{v}Q(YY^{\ast}).\]
	
	Since the relationship is for $Q(YY^{\ast})$ a polynomial in the entries of $YY^{\ast}$, we can find a unique polynomial $P$ with $P(Y,\overline{Y})=Q(YY^{\ast})$, and more generally $P(X,Y)=Q(X{}^{t}Y)$. So far we've shown that this polynomial $P$ is homogeneous (of degree $v$). To show that the original polynomial $Q$ is indeed associated to an \underline{harmonic} polynomial $P$, we need the following lemma: \\


	\begin{lemma} \label{intlemma} Given a homogeneous polynomial $P(X,Y)$ in the entries of $X,Y\in M_{n,k}(\C)$ which we identify with $\C^{N}$ ($N=nk$) which satisfies
		
		\begin{equation} \label{eq intlemma} \int_{\C^{N}}e^{-i\langle w,z\rangle}e^{-i\langle z,w\rangle}e^{-\langle w,w\rangle}P(w,\overline{w})\abs{dw}^{2} = (2\pi)^{N}e^{-\langle z,z\rangle}P(-iz,-i\overline{z}) \end{equation}
		
		for every $z\in\C^{N}$, $P$ must be harmonic (in the sense of my definition above). Here, $\langle,\rangle$ is the standard Hermitian form on $\C^{N}$ (which may also be written $\langle w,z\rangle=\Tr (wz^{\ast})$).\end{lemma}
	
	\vspace*{1em}
	
	\begin{proof}[Proof of Lemma \ref{intlemma}] We identify $\C^{N}$ with $\R^{2N}$ and use the mean value relationship for harmonic (the usual definition of harmonic) polynomials on $\R^{2N}$
		
		\[ \int_{\R^{2N}}e^{-2i(x,y)}e^{-(x,x)}F(x)dx=\pi^{N}e^{-(y,y)}F(-iy), \]
		
		which holds for every $y\in\R^{2N}$ for a polynomial $F$ if and only if $F$ is harmonic on $\R^{2N}$. Here, $(,)$ is the standard inner product on $\R^{2N}$. Writing $w=x_{1}+ix_{2}$ and $z=y_{1}+iy_{2}$ with $x_{1},x_{2},y_{1},y_{2}\in\R^{N}$, $x=(x_{1},x_{2}),y=(y_{1},y_{2})\in\R^{2N}$, setting $F(x)=P(x_{1}+ix_{2},x_{1}-ix_{2})$, and recalling that $\abs{dw}^{2}=2^{N}dx_{1}dx_{2}=2^{N}dx$ it follows immediately that equation (\ref{eq intlemma}) holds if and only if $P(x_{1}+ix_{2},x_{1}-ix_{2})$ is harmonic as a function on $\R^{2N}$, i.e. $P(w,\overline{w})$ with $w\in\C^{N}=M_{n,k}(\C)$ satisfies
		
		\begin{equation} \label{eq intlemma2} \sum_{r=1}^{n}\sum_{s=1}^{k}\frac{\partial^{2}}{\partial w_{r,s}\partial \overline{w}r,s}P(w,\overline{w})=0.   \end{equation}
		
		Thinking of $P(w,\overline{w})=P(x_{1}+ix_{2},x_{1}-ix_{2})$ as an analytic function on $\R^{2N}\subset \C^{N}\times\C^{N}$ (embedded via $(x_{1},x_{2})\mapsto (x_{1}+ix_{2},x_{1}-ix_{2})$), we have an equality of real analytic functions in equation (\ref{eq intlemma}) which holds precisely when equation (\ref{eq intlemma2}) holds. We can extend equation (\ref{eq intlemma}) via analytic continuation to all of $\C^{N}\times\C^{N}$. Since the analytic continuation of $P(w,\overline{w})$ is of course $P(X,Y)$, equation (\ref{eq intlemma2}) becomes
		\[  \sum_{r=1}^{n}\sum_{s=1}^{k}\frac{\partial^{2}}{\partial x_{r,s}\partial y_{r,s}}P(X,Y)=0, \]
		which is precisely our definition \ref{def: harmonic} for $P$ to be harmonic.\end{proof}

	\begin{proof}[Returning to the proof of Theorem \ref{thm}] It remains to show that equation (\ref{eq intlemma}) holds for the polynomial $P$ associated to our differential operator. To do this, we use the relation (\ref{eq comm rel}) for $\gamma=\begin{pmatrix} 0 &-1_{n}\\1_{n} & 0\end{pmatrix}$ and the function $F_{0}(Z)=\text{exp}(i\Tr (Y^{\ast}ZY))$ as above. The commutation relation tells us that
		\[ Q\left(\frac{\partial}{\partial Z}\right) \left(\text{det}(-Z)^{-k}F_{0}(-Z^{-1})\right)=\text{det}(-Z)^{-(k+2v)}Q\left(\frac{\partial}{\partial\gamma Z}\right)\left( F_{0}(-Z^{-1})\right).\]
		Just as in Case III above, the left-hand side is equal to
		\[(-2\pi i)^{-nk}i^{\text{deg}(Q)}\int_{M_{n,k}(\C)} e^{i\Tr UY^{\ast}}e^{i\Tr U^{\ast}Y}e^{i\Tr U^{\ast}ZU}P(U,\overline{U})\abs{dU}^{2}. \]
		For the right-hand side, using the homogeneity of $Q$ and computing we find
		
		\begin{equation*}
			\begin{split}
				& \text{det}(-Z)^{-(k+2v)}Q\left(\frac{\partial}{\partial\gamma Z}\right)\left( F_{0}(-Z^{-1})\right)\\
				& =\text{det}(-Z)^{-k}Q\left(\frac{\partial}{\partial Z}\right)\left( F_{0}(-Z^{-1})\right)\\
				& =\text{det}(-Z)^{-k} i^{\text{deg}(Q)}F_{0}(-Z^{-1})Q((-Z^{-1})YY^{\ast}(-Z^{-1})).
			\end{split}
		\end{equation*}
		
		So
		\begin{multline*} (-2\pi i)^{-nk}\int_{M_{n,k}(\C)} e^{i\Tr UY^{\ast}}e^{i\Tr U^{\ast}Y}e^{i\Tr U^{\ast}ZU}P(U,\overline{U})\abs{dU}^{2} 
			\\ 
			=\text{det}(-Z)^{-k} e^{i\Tr (Y^{\ast}(-Z^{-1})Y)}Q((-Z^{-1})YY^{\ast}(-Z^{-1})). \end{multline*}
		
		Since this holds for any value of $Z$, let's take $Z=i\alpha^{2}$ for a positive definite Hermitian matrix $\alpha$. Then we can write the above as follows:
		
		\begin{multline*}
			(-2\pi i)^{-nk}\int_{M_{n,k}(\C)} e^{i\Tr UY^{\ast}}e^{i\Tr U^{\ast}Y}e^{-\Tr \left((\alpha U)^{\ast}(\alpha U)\right)}P(U,\overline{U})\abs{dU}^{2}\\
			=(-i)^{-nk}\text{det}(-\alpha)^{-2k} e^{-\Tr \left((\alpha^{-1} Y)^{\ast}(\alpha^{-1} Y)\right)}Q(i\alpha^{-2}YY^{\ast}i\alpha^{-2}). 
		\end{multline*}
		
		Re-arranging we have
		
		\begin{multline*}
			\int_{M_{n,k}(\C)} e^{i\Tr UY^{\ast}}e^{i\Tr U^{\ast}Y}e^{-\Tr \left((\alpha U)^{\ast}(\alpha U)\right)}P(U,\overline{U})\abs{dU}^{2} \\
			=(2\pi)^{nk}\text{det}(-\alpha)^{-2k} e^{-\Tr \left((\alpha^{-1} Y)^{\ast}(\alpha^{-1} Y)\right)}P\left(-i\alpha^{-1}(\alpha^{-1}Y),-i\hspace*{1ex}\overline{\alpha^{-1}(\alpha^{-1}Y)}\right). 
		\end{multline*}
		
		Finally, we can change variables to $W=-\alpha U$ and $V=\alpha^{-1}Y$, and since $\abs{d U}^{2}=\text{det}(-\alpha)^{-2k}\abs{d W}^{2}$ we have
		\begin{multline*}
			\int_{M_{n,k}(\C)} e^{-i\Tr WV^{\ast}}e^{-i\Tr W^{\ast}V}e^{-\Tr \left(W^{\ast}W\right)}P(\alpha^{-1}W,\overline{\alpha^{-1}W})\abs{dW}^{2} \\ =(2\pi)^{nk}e^{-\Tr \left(V^{\ast}V\right)}P\left(-i\alpha^{-1}V,-i\hspace*{1ex}\overline{\alpha^{-1}V}\right). 
		\end{multline*}
		
		So by Lemma \ref{intlemma}, $P(\alpha^{-1}X,\alpha^{-1}Y)$ satisfies (\ref{eq intlemma}) and is thus harmonic. We already showed $P$ was homogeneous of degree $v$, so as $\alpha$ is Hermitian it follows that $P(X,Y)$ is harmonic, and indeed pluriharmonic by Remark \ref{rem: harmonic}, concluding the proof of Theorem \ref{thm}. \end{proof}

	\newpage 
	
	\section{Explicit Description of Bilinear Differential Operators}

	In this section we shall state and prove Theorem \nameref{thm b}. That is, we shall explicitly describe differential operators in the case $r=2$, i.e. the space $\mathcal{H}_{n,v}(k_{1},k_{2})$, for a fixed CM field $L/L^{+}$ and $k_{1},k_{2},n,v$ as in the statement of Theorem \ref{thm}. That is, we describe the space of polynomials $P(X,Y)$ on entries $X,Y\in M_{n,k}$ (where $k=k_{1}+k_{2}$) satisfying the following three conditions:
	\[ (\ast) \hspace*{5ex} P(AX,BY)=\text{det}(A)^{v}\text{det}(B)^{v}P(X,Y)\hspace*{3ex}\forall A,B\in\text{GL}_{n},\]
	\[ (\ast\ast) \hspace*{5ex} P(XC^{-1},Y{}^{t}C) =P(X,Y) \hspace*{3ex}\forall C\in\text{GL}_{k_{1}}\oplus\text{GL}_{k_{2}}\subset\text{GL}_{k},\]
	where $\text{GL}_{k_{1}}\oplus\text{GL}_{k_{2}}$ sits inside $\text{GL}_{k}$ as the block-diagonal embedding, and the pluriharmonic condition
	\[ (\ast\ast\ast) \hspace*{5ex} \Delta_{s,t}P(X,Y)=\sum_{u=1}^{k}\frac{\partial^{2}}{\partial x_{s,u}\partial y_{t,u}} P(X,Y)=0 \hspace*{3ex} \forall 1\leq s,t\leq n.\]
	Note that condition $(\ast\ast)$ is equivalent to being able to write $P(X,Y)$ as a polynomial $Q(X_{1}{}^{t}Y_{1},X_{2}{}^{t}Y_{2})$ in the entries of two $n\times n$ matrices, where $X_{1}$ and $Y_{1}$ are the first $k_{1}$ columns of $X$ and $Y$ respectively, $X_{2}$ and $Y_{2}$ the last $k_{2}$ columns of $X$ and $Y$ respectively. In this case, writing $W$ and $Z$ for the $n\times n$ variables of the polynomial $Q$, condition $(\ast)$ can be reformulated as
	\[ (\ast) \hspace*{5ex} Q(AW{}^{t}B,AZ{}^{t}B) = \text{det}(A)^{v}\text{det}(B)^{v}Q(W,Z). \]
	Up to isomorphism (so to match up with Shimura's description in \cite{shimura}) this is equivalent to:
	\[ (\ast) \hspace*{5ex} Q({}^{t}AWB,{}^{t}AZB) = \text{det}(A)^{v}\text{det}(B)^{v}Q(W,Z). \]
	As in definition \ref{def: notation} we denote by $\mathcal{H}_{n,v}(k_{1},k_{2})$ the space of such polynomials $Q$ corresponding to polynomials $P$ satisfying properties $(\ast)$, $(\ast\ast)$ and $(\ast\ast\ast)$. This sits inside the space $\mathcal{Q}_{n,v}(2)$ of polynomials $Q$ satisfying $(\ast)$ (and $(\ast\ast)$ by virtue of being polynomials in $X_{1}{}^{t}Y_{1},X_{2}{}^{t}Y_{2}$). Let $\mathcal{Q}_{n}(2)=\oplus_{v\geq 0}\mathcal{Q}_{n,v}(2)$. Note that this is a graded ring.
	
	\vspace*{1ex}
	
	The notation used for monomials here is $X^{\alpha}Y^{\beta}$ or $W^{l}Z^{m}$ where $\alpha$, $\beta$, $l$ and $m$ are ``index matrices'' with entries in $\Z_{\geq 0}$. More precisely, for a matrix variable $X\in M_{n,k}$ and $\alpha\in M_{n,k}(\Z_{\geq 0})$
	\[ X^{\alpha} = \prod_{i,j}x_{ij}^{\alpha_{ij}},\]
	and similarly for $Y^{\beta}$, $W^{\ell}$ and $Z^{m}$. 
	\newpage
	Recall that 
	\[ \text{det}(W) = \sum_{\sigma\in S_{n}}\text{sign}(\sigma)W^{\sigma}, \]
	where by an abuse of notation I write $\sigma$ for both the permutation in $S_{n}$ and the corresponding permutation matrix, i.e. $\sigma_{ij}=1$ if $\sigma(i)=j$, and $\sigma_{ij}=0$ otherwise.
	
	\vspace*{1ex}
	
	Define polynomials $Q_{a}(W,Z)$ for $a=0,\dots,n$ by
	\[ \text{det}(W+\lambda Z)=\sum_{a=0}^{n}Q_{a}(W,Z)\lambda^{a}. \]
	The goal is to eventually show that the space $\mathcal{Q}_{n}(2)$ is generated by the polynomials $\{Q_{0},\dots,Q_{n}\}$ as a $\C$-algebra. We begin with the following initial result:
	
	\vspace*{1ex}
	
	\begin{lemma}\label{lemma Qn1}
		The polynomials $Q_{0},\dots,Q_{n}$ are a $\C$-basis for $\mathcal{Q}_{n,1}(2)$.
	\end{lemma}
	\begin{proof}
		First, the polynomials $Q_{0},\dots,Q_{n}$ must be linearly independent over $\C$ since they share no monomials. To prove that they span $\mathcal{Q}_{n,1}(2)$, we show that any non-zero $Q\in\mathcal{Q}_{n,1}(2)$ has some multiple of one of the polynomials $Q_{0},\dots,Q_{n}$ as a summand. For any such $Q$, take a monomial with a non-zero coefficient, $C_{\ell,m}W^{\ell}Z^{m}$. Since $Q\in \mathcal{Q}_{n,1}(2)$, $\ell+m$ must be a permutation matrix in order for $Q$ to satisfy the homogeneity condition $(\ast)$ under the action of diagonal matrices. Let $a$ be the number of non-zero entries of $m$ (so $0\leq a \leq n)$. Also by the homogeneity condition considered for permutation matrices, we know that $Q$ in fact has a summand
		\[ \sum_{\sigma\in S_{n}}\sum_{\tau\in S_{n}}\frac{C_{\ell,m}}{a!(n-a)!}\text{sign}(\sigma)\text{sign}(\tau)W^{\sigma\ell\tau}Z^{\sigma m\tau}.\]
		Observing that we may write 
		\[ Q_{a}(W,Z) = \sum_{\sigma,\tau\in S_{n}} \frac{1}{a!(n-a)!}\text{sign}(\sigma)\text{sign}(\tau)W^{\sigma^{-1}(I_{n-a}\oplus 0_{a})\tau}Z^{\sigma^{-1}(0_{n-a}\oplus I_{a})\tau},\]
		we see that the summand of $Q$ we found above is (a re-ordering of) the polynomial $C_{\ell,m}Q_{a}$.
		
	\end{proof}

	\vspace*{3ex}
	
	Since all of the polynomials $Q_{0},\dots,Q_{n}\in \mathcal{Q}_{n,1}(2)$, it follows that for any $v\geq 1$, any product of $v$ of $Q_{0},\dots,Q_{n}$ lies inside $\mathcal{Q}_{n,v}(2)$, that is $\C[Q_{0},\dots,Q_{n}]_{v}\subseteq \mathcal{Q}_{n,v}(2)$ and $\C[Q_{0},\dots,Q_{n}]\subseteq \mathcal{Q}_{n}(2)$. To show that these containments are in fact equalities, we show that the $v$-products of $Q_{0},\dots,Q_{n}$ are a basis for $\mathcal{Q}_{n,v}(2)$ for each $v$. Their linear independence is a consequence of the next lemma, then we show that $\mathcal{Q}_{n,v}(2)$ is of the correct dimension, specifically ${n+v}\choose{v}$.
	
	\vspace*{1ex}

	\begin{lemma} \label{lemma alg ind}
		The polynomials $Q_{0}(W,Z),Q_{1}(W,Z),\dots,Q_{n}(W,Z)$ are algebraically independent in $\C[W,Z]$.
	\end{lemma}
	
	\vspace*{1ex}
	
	\begin{proof}
		Recall that we defined the polynomials $Q_{a}$ (for each $a=0,\dots,n$) by
		\[ \text{det}(W+\lambda Z)=\sum_{a=0}^{n}Q_{a}(W,Z)\lambda^{a}. \]
		In particular, evaluating at $Z=1_{n}$ we have that
		\[ \text{det}(W+\lambda 1_{n})=\sum_{a=0}^{n}Q_{a}(W,1_{n})\lambda^{a} \]
		and therefore $Q_{a}(W,1_{n})$ is the $(n-a)$th elementary symmetric polynomial in the eigenvalues of $W$ (counted with multiplicity). Explicitly, if $W$ is diagonal then we have
		\[ \sum_{a=0}^{n}s_{n-a}(w_{1,1},w_{2,2},\dots,w_{n,n})\lambda^{a} = \prod_{i=1}^{n}(w_{i,i}+\lambda) = \text{det}(W+\lambda 1_{n}) = \sum_{a=0}^{n}Q_{a}(W,1_{n})\lambda^{a}, \]
		where $s_{n-a}$ is the $(n-a)$th elementary symmetric polynomial on $n$ arguments (and $s_{0}$ is the constant polynomial $1$). So any non-trivial algebraic relation on $Q_{0},\dots,Q_{n}$ restricts to an algebraic relation on $s_{0},\dots,s_{n}$ inside the ring of polynomials in $n$ variables (given by the eigenvalues of $W$). But the elementary symmetric polynomials are algebraically independent, so the polynomials $Q_{0},\dots,Q_{n}$ must be too.
	\end{proof}

	\vspace*{3ex}
	
	\begin{proposition}\label{prop dimension}
		The dimension of $\mathcal{Q}_{n,v}(2)$ as a vector space over $\C$ is ${n+v}\choose{v}$ for each integer $v>0$. (The same holds for $v=0$ trivially.)
	\end{proposition}
	
	\vspace*{2ex}
	
	\begin{proof}
		First, we have that $\C[W,Z]\cong \C[W]\otimes\C[Z]$ not only as $\C$-algebras but also as $\text{GL}_{n}\times\text{GL}_{n}$ representations. Now, $\mathcal{Q}_{n,v}(2)\subset\C[W,Z]_{nv}$ since, by considering the action of diagonal matrices, for example, or by previous arguments, each polynomial $Q\in\mathcal{Q}_{n,v}(2)$ is homogeneous of total degree $nv$ in $W$ and $Z$. The standard decomposition of tensor products is
		\[ \C[W,Z]_{nv}\cong\bigoplus_{a=0}^{nv}\C[W]_{a}\otimes\C[Z]_{nv-a}.\]
		Now, the space $\C[W]_{a}$ together with its $\text{GL}_{n}\times\text{GL}_{n}$-action $(A,B)f(W)=f({}^{t}AWB)$ is precisely the representation $\tau^{a}$ of Shimura \cite{shimura} \S 12, and similarly $\C[Z]_{nv-a}$ with the same action is Shimura's $\tau^{nv-a}$. Since $\mathcal{Q}_{n,v}(2)$ is precisely the summand of $\text{det}^{v}\boxtimes\text{det}^{v}$ inside $\C[W,Z]_{nv}$, we just need to find the summands (or, for the dimension, the multiplicity of) the representation $\text{det}^{v}\boxtimes\text{det}^{v}$ in all of the summands $\C[W]_{a}\otimes\C[Z]_{nv-a}=\tau^{a}\otimes\tau^{nv-a}$ ($a=0,\dots,nv$).
		
		\vspace*{2ex}
		
		By \cite{shimura} Theorem 12.7 (attributed to L.-K. Hua), for any irreducible representations $\rho_{1}$ and $\rho_{2}$ of $\text{GL}_{n}$, the representation $\rho_{1}\boxtimes\rho_{2}$ occurs as a summand in the representation $\tau^{a}$ with multiplicity one if and only if $\rho_{1}$ and $\rho_{2}$ correspond to the same partition $\lambda$ of $a$ (with at most $n$ parts), $\lambda=(\lambda_{1},\dots,\lambda_{n})$ with $\lambda_{1}\geq\lambda_{2}\geq\dots\geq \lambda_{n}\geq 0$ and $\lambda_{1}+\dots+\lambda_{n}=a$, and does not occur otherwise. For representations $\rho_{1}$ and $\rho_{2}$ both corresponding to the same $n$-part partition $\lambda$ of $a$, and representations $\sigma_{1}$ and $\sigma_{2}$ both corresponding to the same $n$-part partition $\mu$ of $nv-a$, the multiplicity of $\text{det}^{v}\boxtimes\text{det}^{v}$ inside $(\rho_{1}\boxtimes\rho_{2})\otimes(\sigma_{1}\boxtimes\sigma_{2})$ is the product of the multiplicity of $\text{det}^{v}$ in $\rho_{1}\otimes\sigma_{1}$ and $\rho_{2}\otimes\sigma_{2}$ respectively. 
		
		Since this multiplicity depends only on the partitions $\lambda$ and $\mu$, the representation $\text{det}^{v}$ occurs with the same multiplicity within $\rho_{1}\otimes\sigma_{1}$ and $\rho_{2}\otimes\sigma_{2}$. This multiplicity is precisely given by the Littlewood--Richardson coefficient $\text{LR}_{\lambda\mu}^{(v^{n})}$, where $(v^{n})$ denotes the partition $(v,v,\dots,v)$ of $nv$. Either by application of the Littlewood--Richardson rule (as stated in Chapter 5 of \cite{Fulton} for example), or as a consequence of a result of Okada (stated as a remark following Theorem 2.1 of \cite{Okada}) the coefficient $\text{LR}_{\lambda\mu}^{(v^{n})}=1$ if and only if $\lambda_{1}\leq v$ and $\mu_{i}+\lambda_{n+1-i}=v$ for each $i=1,\dots,n$, and $\text{LR}_{\lambda\mu}^{(v^{n})}=0$ otherwise.
		
		\vspace*{2ex}
		
		So, for each $a=0,\dots,nv$ and each partition $\lambda=(\lambda_{1},\dots,\lambda_{n})$ of $a$ with $\lambda_{1}\leq v$ there is exactly one copy of $\text{det}^{v}\boxtimes\text{det}^{v}$ inside $\tau^{a}\otimes\tau^{nv-a}$. That is, the multiplicity of $\text{det}^{v}\boxtimes\text{det}^{v}$ inside $\tau^{a}\otimes\tau^{nv-a}$ is equal to the number of Young diagrams with $a$ boxes fitting inside $(v)^{n}$, the $v\times n$ rectangle. So the total multiplicity of $\text{det}^{v}\otimes\text{det}^{v}$ in $\C[W,Z]_{nv}$ is equal to the number of Young diagrams of any size inside $(v)^{n}$. By drawing a lattice path along the lower-right edge of any such Young diagram, we see that this is equal to the number of lattice paths from $(0,0)$ to $(v,n)$. This is equal to ${n+v}\choose{v}$ since the path is $n+v$ steps long, and we just need to choose $v$ of those steps at which to travel right or equivalently $n$ of the steps at which to travel upwards.

	\end{proof}

	\begin{corollary}\label{coro Qnv}
		We have
		\[ \mathcal{Q}_{n}(2)=\C[Q_{0},\dots,Q_{n}] \hspace*{3ex} \text{ and } \hspace*{3ex} \mathcal{Q}_{n,v}(2)=\C[Q_{0},\dots,Q_{n}]_{v}. \]
	\end{corollary}
	
	\begin{proof}
		By Lemma \ref{lemma Qn1}, the polynomials $Q_{0},\dots,Q_{n}$ are a basis for $\mathcal{Q}_{n,1}(2)$ as a vector space over $\C$. Any $v$-product of the polynomials $Q_{0},\dots,Q_{n}$ is an element of $\mathcal{Q}_{n,v}(2)$. By Lemma \ref{lemma alg ind}, the polynomials $Q_{0},\dots,Q_{n}$ are algebraically independent and therefore all the distinct $v$-products of $Q_{0},\dots,Q_{n}$ are linearly independent over $\C$. There are ${n+v}\choose{v}$ such products, and by Proposition \ref{prop dimension} this is the dimension of $ \mathcal{Q}_{n,v}(2)$, so the set of $v$-products of the polynomials $Q_{0},\dots,Q_{n}$ is a basis for $\mathcal{Q}_{n,v}(2)$ for every degree $v$.
		
	\end{proof}
	
	\vspace*{1ex}
	
	\begin{remark}\label{rem Q and P}
		A priori, it is unclear that the space of polynomials inside $\mathcal{Q}_{n,v}(2)$ which arise as polynomials of the form $P(X,Y)$ is necessarily all of $\mathcal{Q}_{n,v}(2)$. That is, not every polynomial $Q\in\mathcal{Q}_{n,v}(2)$ might correspond to a polynomial $P(X,Y)$. However, having shown that $Q_{0},\dots,Q_{n}$ generate $\mathcal{Q}_{n}(2)$ then we see that this is indeed the case since, setting $W=X_{1}{}^{t}Y_{1}$ and $Z=X_{2}{}^{t}Y_{2}$ as usual,
		\[ \sum_{m=0}^{n}Q_{m}(W,Z)\lambda^{m} = \text{det}(W+\lambda Z) = \text{det}(X_{1}{}^{t}Y_{1} + \lambda X_{2}{}^{t}Y_{2}) = \sum_{m=0}^{n}P_{m}(X,Y)\lambda^{m},\]
		for some polynomials $P_{m}(X,Y)$. By comparing coefficients of $\lambda^{m}$, we must have that $P_{m}(X,Y)=Q_{m}(X_{1}{}^{t}Y_{1},X_{2}{}^{t}Y_{2})$.  
	\end{remark}

	\begin{theorem}\label{thm coeff}
		For each fixed $n$ and $v$, writing a polynomial $Q\in\mathcal{Q}_{n,v}(2)$ as
		\[ Q=\sum_{\alpha}C(\alpha)\prod_{j=0}^{n}Q_{j}^{\alpha_{j}},\]
		where the sum ranges over index tuples $\alpha=(\alpha_{0},\dots,\alpha_{n})\in\Z_{\geq 0}^{n+1}$ with $\sum_{j=0}^{n}\alpha_{j}=v$, $Q$ is pluriharmonic i.e. $Q\in\mathcal{H}_{n,v}(k_{1},k_{2})$ if the coefficients $C(\alpha)$ with $\alpha\neq(0,\dots,0,v)$ satisfy the following linear relations:
		
		\begin{equation*}
			\begin{split} 
				0 = & C(\alpha)\alpha_{m}(k_{1}+1-n+m)  + C(\alpha-1_{m}+1_{m+1}) (\alpha-1_{m}+1_{m+1})_{m+1}(k_{2}-m) \\
				& + C(\alpha)\alpha_{m} (\alpha_{m}-1) \\
				& + \sum_{\substack{m<\ell\leq \ell'\leq n\\\ell+\ell'-m-1\leq n}} C(\tilde{\alpha}(m,\ell,\ell')) \tilde{\alpha}(m,\ell,\ell')_{\ell} (\tilde{\alpha}(m,\ell,\ell')_{\ell'}-\delta_{\ell,
					\ell'}) (2-\delta_{\ell,\ell'})\\
				& - \sum_{\substack{m<\ell\leq \ell'\leq n\\\ell+\ell'-m\leq n}} C(\alpha(m,\ell,\ell'))\alpha(m,\ell,\ell')_{\ell} (\alpha(m,\ell,\ell')_{\ell'}-\delta_{\ell,\ell'}) (2-\delta_{\ell,\ell'}). \\
			\end{split}
		\end{equation*}
		
		Here $0\leq m<n$ is the least integer such that $\alpha_{m}>0$ and we write $\alpha(m,\ell,\ell'):=\alpha-1_{m}+1_{\ell}+1_{\ell'}-1_{\ell+\ell'-m}$ and $\tilde{\alpha}(m,\ell,\ell'):=\alpha-1_{m}+1_{\ell}+1_{\ell'}-1_{\ell+\ell'-m-1}$.
		In particular, every coefficient $C(\alpha)$ is determined by the value of $C(0,\dots,0,v)$ so there is a unique (up to scaling) pluriharmonic polynomial inside $\mathcal{Q}_{n,v}(2)$, $\mathcal{H}_{n,v}(k_{1},k_{2})$ is one-dimensional.
		
	\end{theorem}

	\begin{proof} We begin by rephrasing $(\ast\ast\ast)$, restated below, in terms of the polynomial $Q(W,Z)$ associated to $P(X,Y)$. 
		\[ (\ast\ast\ast) \hspace*{5ex} \Delta_{i,j}P(X,Y)=\sum_{u=1}^{k}\frac{\partial^{2}}{\partial x_{i,u}\partial y_{j,u}} P(X,Y)=0 \hspace*{3ex} \forall 1\leq i,j\leq n.\]
		First we re-write the operator $\Delta_{i,j}$ (in terms of $X,Y$ acting on $P(X,Y)$) as a differential operator in terms of $W=X_{1}^{t}Y_{1}$ and $Z=X_{2}^{t}Y_{2}$ acting on the polynomial $Q(W,Z)$ associated to $P$ (so $P(X,Y)=Q(W,Z)$). We have
		\[ w_{i,j}=\sum_{u=1}^{k_{1}}x_{i,u}y_{j,u} \hspace*{5ex} z_{i,j}=\sum_{u=k_{1}+1}^{k}x_{i,u}y_{j,u}, \]
		\[ \sum_{u=1}^{k}\frac{\partial}{\partial y_{j,u}} = \sum_{s=1}^{n}\left( \sum_{u=1}^{k_{1}} x_{s,u}\frac{\partial}{\partial w_{s,j}} + \sum_{u=k_{1}+1}^{k} x_{s,u}\frac{\partial}{\partial z_{s,j}} \right), \]
		and then
		\begin{equation*}
			\begin{split}
				&\sum_{u=1}^{k}\frac{\partial}{\partial x_{i,u}}\frac{\partial}{\partial y_{j,u}}\\ & = k_{1}\frac{\partial}{\partial w_{i,j}} + \sum_{s,t=1}^{n}\sum_{u=1}^{k_{1}}x_{s,u}y_{t,u}\frac{\partial}{\partial w_{s,j}}\frac{\partial}{\partial w_{i,t}} + k_{2}\frac{\partial}{\partial z_{i,j}} + \sum_{s,t=1}^{n}\sum_{u=k_{1}+1}^{k}x_{s,u}y_{t,u}\frac{\partial}{\partial z_{s,j}}\frac{\partial}{\partial z_{i,t}} \\
				& =  k_{1}\frac{\partial}{\partial w_{i,j}} + \sum_{s,t=1}^{n}w_{s,t}\frac{\partial}{\partial w_{s,j}}\frac{\partial}{\partial w_{i,t}} + k_{2}\frac{\partial}{\partial z_{i,j}} + \sum_{s,t=1}^{n}z_{s,t}\frac{\partial}{\partial z_{s,j}}\frac{\partial}{\partial z_{i,t}}.
			\end{split}
		\end{equation*}
		We write 
		\[ L^{(k_{1})}_{i,j}:=k_{1}\frac{\partial}{\partial w_{i,j}} + \sum_{s,t=1}^{n}w_{s,t}\frac{\partial}{\partial w_{s,j}}\frac{\partial}{\partial w_{i,t}}\hspace*{3ex}\text{and}\hspace*{3ex} L'^{(k_{2})}_{i,j}:=k_{2}\frac{\partial}{\partial z_{i,j}} + \sum_{s,t=1}^{n}z_{s,t}\frac{\partial}{\partial z_{s,j}}\frac{\partial}{\partial z_{i,t}}.\]
		So for a polynomial $P$ satisfying $(\ast\ast)$, $\Delta_{i,j}(P(X,Y))=0$ if and only if $(L^{(k_{1})}_{i,j}+L'^{(k_{2})}_{i,j})(Q(W,Z))=0$ where $Q$ is the polynomial associated to $P$. Recall that for a polynomial $P$ (and associated $Q$) satisfying $(\ast)$ and $(\ast\ast)$, $(\ast\ast\ast)$ holds if and only if $\sum_{i=1}^{n}\Delta_{i,i}(P(X,Y))=0$.
		
		Therefore, to find polynomials in $\mathcal{H}_{n,v}(k_{1},k_{2})$ we should calculate the action of $L_{i,i}^{(k_{1})}$ and $L'^{(k_{2})}_{i,i}$ on $Q_{0},\dots,Q_{n}$ and on products. First, for $Q,Q'\in\mathcal{Q}_{n}(2)$ we have
		\[ L_{i,i}^{(k_{1})}(QQ') = L_{i,i}^{(k_{1})}(Q)\cdot Q' + Q\cdot L_{i,i}^{(k_{1})} (Q') + \sum_{s,t=1}^{n} w_{s,t} \left( \frac{\partial}{\partial w_{s,i}} Q \cdot \frac{\partial}{\partial w_{i,t}} Q' + \frac{\partial}{\partial w_{i,t}} Q \cdot \frac{\partial}{\partial w_{s,i}} Q' \right),
		\]
		and
		\[ L'^{(k_{2})}_{i,i}(QQ') = L'^{(k_{2})}_{i,i}(Q)\cdot Q' + Q\cdot L'^{(k_{2})}_{i,i}(Q') + \sum_{s,t=1}^{n} z_{s,t} \left( \frac{\partial}{\partial z_{s,i}} Q \cdot \frac{\partial}{\partial z_{i,t}} Q' + \frac{\partial}{\partial z_{i,t}} Q \cdot \frac{\partial}{\partial z_{s,i}} Q' \right).
		\]
		\newpage
		\noindent For ease of notation, we use the following:
		\[ (Q,Q')_{i,W}:= \sum_{s,t=1}^{n} w_{s,t} \left( \frac{\partial}{\partial w_{s,i}} Q \cdot \frac{\partial}{\partial w_{i,t}} Q' + \frac{\partial}{\partial w_{i,t}} Q \cdot \frac{\partial}{\partial w_{s,i}} Q' \right),\]
		and
		\[ (Q,Q')_{i,Z}: = \sum_{s,t=1}^{n} z_{s,t} \left( \frac{\partial}{\partial z_{s,i}} Q \cdot \frac{\partial}{\partial z_{i,t}} Q' + \frac{\partial}{\partial z_{i,t}} Q \cdot \frac{\partial}{\partial z_{s,i}} Q' \right).\]
		So, re-writing the above, we obtain:
		\[ L_{i,i}^{(k_{1})}(QQ') = L_{i,i}^{(k_{1})}(Q)\cdot Q' + Q\cdot L_{i,i}^{(k_{1})} (Q') + (Q,Q')_{i,W},\]
		and
		\[ L'^{(k_{2})}_{i,i}(QQ') = L'^{(k_{2})}_{i,i}(Q)\cdot Q' + Q\cdot L'^{(k_{2})}_{i,i}(Q') + (Q,Q')_{i,Z}.\]
		Next, we calculate $L^{(k_{1})}_{i,i}(Q_{a})$ and $L'^{(k_{2})}_{i,i}(Q_{a})$ for $a=0,\dots,n$. There are two easy cases: recall that $Q_{0}(W,Z)=\text{det}(W)$ and $Q_{n}(W,Z)=\text{det}(Z)$ so 
		\[ L^{(k_{1})}_{i,i}(Q_{n}) = L'^{(k_{2})}_{i,i}(Q_{0}) = 0.\]
		For the other cases, we compute $L^{(k_{1})}_{i,i}(\text{det}(W+\lambda Z))$ and $L'^{(k_{2})}_{i,i}(\text{det}(W+\lambda Z))$ respectively, and consider the coefficients of $\lambda^{a}$. 
		
		Here we'll introduce some notation. For an $n\times n$ matrix $A$, I write $A^{\left[\substack{ i_{1},\dots,i_{\ell}\\j_{1},\dots,j_{\ell}}\right]}$ for the $(n-\ell)\times(n-\ell)$ matrix obtained from $A$ by deleting rows indexed $i_{1},\dots,i_{\ell}$ and columns indexed $j_{1},\dots,j_{\ell}$. Then, we write
		\[ \text{det}\left((W+\lambda Z)^{\left[\substack{ i_{1},\dots,i_{\ell}\\j_{1},\dots,j_{\ell}}\right]}\right) = \sum_{m=0}^{n-\ell}Q_{m}^{\left[\substack{ i_{1},\dots,i_{\ell}\\j_{1},\dots,j_{\ell}}\right]}(W,Z)\lambda^{m}. \]
		That is, $Q_{m}^{\left[\substack{ i_{1},\dots,i_{\ell}\\j_{1},\dots,j_{\ell}}\right]}(W,Z)$ is the coefficient of $\lambda^{m}$ in $\text{det}(W+\lambda Z)^{\left[\substack{ i_{1},\dots,i_{\ell}\\j_{1},\dots,j_{\ell}}\right]}$ (and we take $Q_{m}^{\left[\substack{ i_{1},\dots,i_{\ell}\\j_{1},\dots,j_{\ell}}\right]}(W,Z)=0$ if $m<0$ or $m>n-\ell$). Now:
		\begin{equation*}
			\begin{split}
				L^{(k_{1})}_{i,i}(\text{det}(W+\lambda Z)) & = k_{1}\frac{\partial}{\partial w_{i,i}}\text{det}(W+\lambda Z) + \sum_{s=1}^{n}\sum_{t=1}^{n}w_{s,t}	\frac{\partial}{\partial w_{s,i}}\frac{\partial}{\partial w_{i,t}}\text{det}(W+\lambda Z) \\
				& = k_{1}\text{det}(W+\lambda Z)^{\left[\substack{ i\\i}\right]} + w_{i,i}\frac{\partial}{\partial w_{i,i}}\text{det}(W+\lambda Z)^{\left[\substack{ i\\i}\right]} \\
				& + \sum_{\substack{s=1\\s\neq i}}^{n} w_{s,i}\frac{\partial}{\partial w_{s,i}}\text{det}(W+\lambda Z)^{\left[\substack{ i\\i}\right]} + \sum_{\substack{t=1\\t\neq i}}^{n} w_{i,t}\frac{\partial}{\partial w_{i,i}}(-1)^{i+t}\text{det}(W+\lambda Z)^{\left[\substack{ i\\t}\right]} \\
				& + \sum_{\substack{s=1\\s\neq i}}^{n} \sum_{\substack{t=1\\t\neq i}}^{n}w_{s,t}\frac{\partial}{\partial w_{s,i}}(-1)^{i+t}\text{det}(W+\lambda Z)^{\left[\substack{ i\\t}\right]}.
			\end{split}
		\end{equation*}
		Notice that in the middle three terms in the sum above we deleted the $i$th row or the $i$th column, so those partial derivatives of determinants are all zero. So
		\begin{equation*}
			\begin{split}
				L^{(k_{1})}_{i,i}(\text{det}(W+\lambda Z)) & = k_{1}\text{det}(W+\lambda Z)^{\left[\substack{ i\\i}\right]} +  \sum_{\substack{s=1\\s\neq i}}^{n} \sum_{\substack{t=1\\t\neq i}}^{n}w_{s,t}\frac{\partial}{\partial w_{s,i}}(-1)^{i+t}\text{det}(W+\lambda Z)^{\left[\substack{ i\\t}\right]} \\
				& = k_{1}\text{det}(W+\lambda Z)^{\left[\substack{ i\\i}\right]} -  \sum_{\substack{s=1\\s\neq i}}^{n} \sum_{\substack{t=1\\t\neq i}}^{n}w_{s,t}(-1)^{s+t+\delta_{s>i}+\delta_{t>i}}\text{det}(W+\lambda Z)^{\left[\substack{ s,i\\t,i}\right]},
			\end{split}
		\end{equation*}
		where $\delta_{s>i}=1$ if $s>i$ and $0$ otherwise, and above we use that for $i\neq t$ we have $\delta_{i>t}+\delta_{t>i}=1$. We compute:
		\begin{equation*}
			\begin{split}
				\sum_{\substack{s=1\\s\neq i}}^{n} \sum_{\substack{t=1\\t\neq i}}^{n}w_{s,t}(-1)^{s+t+\delta_{s>i}+\delta_{t>i}} & \text{det}(W+\lambda Z)^{\left[\substack{ s,i\\t,i}\right]} \\  = & \sum_{\substack{s=1\\s\neq i}}^{n} \sum_{\substack{t=1\\t\neq i}}^{n}(w_{s,t}+\lambda z_{s,t})(-1)^{s+t+\delta_{s>i}+\delta_{t>i}}\text{det}(W+\lambda Z)^{\left[\substack{ s,i\\t,i}\right]} \\
				& - \sum_{\substack{s=1\\s\neq i}}^{n} \sum_{\substack{t=1\\t\neq i}}^{n}\lambda z_{s,t}(-1)^{s+t+\delta_{s>i}+\delta_{t>i}}\text{det}(W+\lambda Z)^{\left[\substack{ s,i\\t,i}\right]} \\
				= & (n-1)\text{det}(W+\lambda Z)^{\left[\substack{i\\i}\right]} -\lambda \frac{d}{d\lambda}\left(\text{det}(W+\lambda Z)^{\left[\substack{i\\i}\right]}\right).
			\end{split}
		\end{equation*}
		That is,
		\[ 	L^{(k_{1})}_{i,i}(\text{det}(W+\lambda Z)) =  k_{1}\text{det}(W+\lambda Z)^{\left[\substack{ i\\i}\right]} - (n-1)\text{det}(W+\lambda Z)^{\left[\substack{i\\i}\right]} +\lambda \frac{d}{d\lambda}\left(\text{det}(W+\lambda Z)^{\left[\substack{i\\i}\right]}\right). \]
		Comparing coefficients of $\lambda^{a}$ (for $a=0,\dots,n-1$) we find
			\begin{equation} \label{eq4}  L^{(k_{1})}_{i,i}(Q_{a}(W,Z)) = (k_{1}+1-n+a)Q^{\left[\substack{i\\i}\right]}_{a}(W,Z). \end{equation}
		Similarly, $L'^{(k_{2})}_{i,i}(\text{det}(W+\lambda Z))$ is equal to the following:
		\begin{equation*} 
			\begin{split}
				&k_{2}\frac{\partial}{\partial z_{i,i}}\text{det}(W+\lambda Z) + \sum_{s=1}^{n}\sum_{t=1}^{n}z_{s,t} \frac{\partial}{\partial z_{s,i}} \frac{\partial}{\partial z_{i,t}}\text{det}(W+\lambda Z) \\
				& = k_{2} \lambda \text{det}(W+\lambda Z)^{\left[\substack{i\\i}\right]} - \sum_{\substack{s=1\\s\neq i}}^{n}\sum_{\substack{t=1\\t\neq i}}^{n}z_{s,t}(-1)^{s+t+\delta_{s>i}+\delta_{t>i}}\lambda^{2}\text{det}(W+\lambda Z)^{\left[\substack{ s,i\\t,i}\right]} \\
				& =  k_{2} \lambda \text{det}(W+\lambda Z)^{\left[\substack{i\\i}\right]} - \lambda^{2}\frac{d}{d\lambda} \left(  \text{det}(W+\lambda Z)^{\left[\substack{i\\i}\right]} \right).
			\end{split}
		\end{equation*}
		Again by comparing coefficients of $\lambda^{a}$ (for $a=1,\dots,n$) we get that
		\begin{equation}\label{eq5} L'^{(k_{2})}_{i,i}(Q_{a}(W,Z)) = (k_{2}+1-a)Q^{\left[\substack{i\\i}\right]}_{a-1}(W,Z). \end{equation}
		Next, for $0\leq a,b\leq n$, we compute: 
		\begin{equation} \label{eq6}
			\begin{split}
				(Q_{a},Q_{b})_{i,W} & = \sum_{s,t=1}^{n} w_{s,t}\left( (-1)^{s+i}Q^{\left[\substack{s\\i}\right]}_{a} (-1)^{i+t} Q^{\left[\substack{i\\t}\right]}_{b} + (-1)^{i+t} Q^{\left[\substack{i\\t}\right]}_{a} (-1)^{s+i} Q^{\left[\substack{s\\i}\right]}_{b} \right)\\
				& = \sum_{s,t=1}^{n}(-1)^{s+t}w_{s,t}\left( Q^{\left[\substack{s\\i}\right]}_{a} Q^{\left[\substack{i\\t}\right]}_{b} + Q^{\left[\substack{i\\t}\right]}_{a} Q^{\left[\substack{s\\i}\right]}_{b} \right),
			\end{split}
		\end{equation}
		and
		\begin{equation}\label{eq7}
			\begin{split}
				(Q_{a},Q_{b})_{i,Z} & = \sum_{s,t=1}^{n} z_{s,t}\left( (-1)^{s+i}Q^{\left[\substack{s\\i}\right]}_{a-1} (-1)^{i+t} Q^{\left[\substack{i\\t}\right]}_{b-1} + (-1)^{i+t} Q^{\left[\substack{i\\t}\right]}_{a-1} (-1)^{s+i} Q^{\left[\substack{s\\i}\right]}_{b-1} \right)\\
				& = \sum_{s,t=1}^{n}(-1)^{s+t}z_{s,t}\left( Q^{\left[\substack{s\\i}\right]}_{a-1} Q^{\left[\substack{i\\t}\right]}_{b-1} + Q^{\left[\substack{i\\t}\right]}_{a-1} Q^{\left[\substack{s\\i}\right]}_{b-1} \right).
			\end{split}
		\end{equation}
		Notice that
		\begin{equation}\label{eq8} \sum_{s=1}^{n}(-1)^{s+t}(w_{s,t}+\lambda z_{s,t})\text{det}(W+\lambda Z)^{\left[\substack{s\\i}\right]}=\delta_{i,t}(-1)^{i+t}\text{det}(W+\lambda Z)=\delta_{i,t}\text{det}(W+\lambda Z).\end{equation}
		Here $\delta_{i,t}$ is the usual Kronecker delta. This equality holds since the left hand side is (up to sign) the determinant of the matrix obtained from $W+\lambda Z$ by copying the entries of the $t$-th column to the $i$-th column, which is $0$ if $i\neq t$ and the usual determinant of $W+\lambda Z$ otherwise. Comparing coefficients of $\lambda^{a}$ on each side of (\ref{eq8}) we find
		\[ \sum_{s=1}^{n}(-1)^{s+t}w_{s,t}Q^{\left[\substack{s\\i}\right]}_{a} +\sum_{s=1}^{n}(-1)^{s+t} z_{s,t}Q^{ \left[\substack{s\\i}\right]}_{a-1} = \delta_{i,t}Q_{a},\]
		and
		\[ \sum_{t=1}^{n}(-1)^{s+t}w_{s,t}Q^{\left[\substack{i\\t}\right]}_{a} +\sum_{t=1}^{n}(-1)^{s+t} z_{s,t}Q^{ \left[\substack{i\\t}\right]}_{a-1} = \delta_{s,i}Q_{a}.\]
		In particular,
		\[ \sum_{t=1}^{n}\sum_{s=1}^{n}(-1)^{s+t}w_{s,t}Q^{\left[\substack{s\\i}\right]}_{a}Q^{\left[\substack{i\\t}\right]}_{b} +\sum_{t=1}^{n}\sum_{s=1}^{n}(-1)^{s+t} z_{s,t} Q^{ \left[\substack{s\\i}\right]}_{a-1} Q^{\left[\substack{i\\t}\right]}_{b} = Q_{a}Q^{\left[\substack{i\\i}\right]}_{b},\]
		and
		\[ \sum_{s=1}^{n}\sum_{t=1}^{n}(-1)^{s+t}w_{s,t}Q^{\left[\substack{i\\t}\right]}_{a}Q^{\left[\substack{s\\i}\right]}_{b} +\sum_{s=1}^{n}\sum_{t=1}^{n}(-1)^{s+t} z_{s,t}Q^{ \left[\substack{i\\t}\right]}_{a-1}Q^{\left[\substack{s\\i}\right]}_{b} = Q_{a}Q^{\left[\substack{i\\i}\right]}_{b}.\]
		\vspace*{1em}
		Together with (\ref{eq6}) and (\ref{eq7}) we extract the following:
		\[ (Q_{a},Q_{b})_{i,W} + (Q_{a},Q_{b+1})_{i,Z} = 2 Q_{a} Q^{\left[\substack{i\\i}\right]}_{b} 
		\hspace*{3ex}\text{and}\hspace*{3ex} (Q_{a},Q_{b})_{i,W} + (Q_{a+1},Q_{b})_{i,Z} = 2  Q^{\left[\substack{i\\i}\right]}_{a} Q_{b}, \]
		so
		\begin{equation}\label{eq9} (Q_{a},Q_{b})_{i,W} = 2Q_{a}Q^{\left[\substack{i\\i}\right]}_{b} - 2 Q^{\left[\substack{i\\i}\right]}_{a-1}Q_{b+1} + (Q_{a-1},Q_{b+1})_{i,W},\end{equation}
		and
		\begin{equation}\label{eq10} (Q_{a},Q_{b})_{i,Z} = 2 Q^{\left[\substack{i\\i}\right]}_{a-1}Q_{b} - 2 Q_{a-1} Q^{\left[\substack{i\\i}\right]}_{b} + (Q_{a-1},Q_{b+1})_{i,Z}. \end{equation}
		\phantom{\qedhere}
		\end{proof}
		
		To proceed, we need the following lemma:
		
		\begin{lemma} \label{lemma lin indep}
		For any $n\geq 1$, the polynomials
		\[ \sum_{i=1}^{n}Q^{\left[\substack{i\\i}\right]}_{a} \hspace*{5ex} 0\leq a\leq n-1\]
		are linearly independent over $\mathcal{Q}_{n}(2)$.
		\end{lemma}
		
		\begin{proof}[Proof of Lemma \ref{lemma lin indep}]
		
		We use induction on $n$. For $n=1$, the statement is trivially true since the only polynomial $Q^{\left[\substack{1\\1}\right]}_{0}$ is the determinant of the empty matrix i.e. the constant $1$.	
		Suppose the statement holds up to $n-1$, and suppose there exist polynomials $R_{a}$ in $Q_{0},\dots,Q_{n}$ for $0\leq a \leq n-1$ such that
		\[ \sum_{a=0}^{n-1}R_{a}(Q_{0},\dots,Q_{n})\left(\sum_{i=1}^{n}Q^{\left[\substack{i\\i}\right]}_{a}\right) =0. \]
		We show that $R_{a}=0$ for every $0\leq a\leq n-1$. Consider the restriction to $z_{j,n}=z_{n,j}=0$ for every $1\leq j\leq n$ and $w_{j,n}=w_{n,j}=0$ for every $1\leq j\leq n-1$. In this case we have $Q_{n}=0$, $Q_{m}=w_{n,n}Q^{\left[\substack{n\\n}\right]}_{m}$ for $0\leq m\leq n-1$, $Q^{\left[\substack{n\\n}\right]}_{a}$ is unchanged for every $a$, and for $0\leq i\leq n-1$ $Q^{\left[\substack{i\\i}\right]}_{a}$ becomes $w_{n,n}Q^{\left[\substack{i,n\\i,n}\right]}_{a}$. So our equation above restricts to
		\[ \sum_{a=0}^{n-1}R_{a}(w_{n,n}Q^{\left[\substack{n\\n}\right]}_{0},\dots,w_{n,n}Q^{\left[\substack{n\\n}\right]}_{n-1},0) \left( \sum_{i=1}^{n-1}w_{n,n} Q^{\left[\substack{i,n\\i,n}\right]}_{a} + Q^{\left[\substack{n\\n}\right]}_{a}\right) =0. \]
		Comparing terms containing $w_{n,n}$ we have
		\[ \sum_{a=0}^{n-1}R_{a}(w_{n,n}Q^{\left[\substack{n\\n}\right]}_{0},\dots,w_{n,n}Q^{\left[\substack{n\\n}\right]}_{n-1},0) \left( \sum_{i=1}^{n-1}w_{n,n} Q^{\left[\substack{i,n\\i,n}\right]}_{a} \right) =0. \]
		Further restricting to $w_{n,n}=1$, note that we obtain exactly the relation for the case $n-1$ with the top left $(n-1)\times(n-1)$ minors of $W$ and $Z$, $W^{\left[\substack{n\\n}\right]}$ and $Z^{\left[\substack{n\\n}\right]}$. So we must have that $R_{a}(Q_{0},\dots,Q_{n-1},0)=0$ for every $a$, and thus $R_{a}(Q_{0},\dots,Q_{n-1},Q_{n})$ is divisible by $Q_{n}$. But we can factor out a suitable power of $Q_{n}$ and repeat the same argument, showing that for each $0\leq a\leq n-1$, $R_{a}(Q_{0},\dots,Q_{n})$ is divisible by arbitrarily large powers of $Q_{n}$ and thus must be zero.
		
		\end{proof}

		\begin{proof}[Returning to the proof of Theorem \ref{thm coeff}]
		We consider the action of the Laplacian operator $\sum_{i=1}^{n}\Delta_{i,i}=\sum_{i=1}^{n}L^{(k_{1})}_{i,i}+L'^{(k_{2})}_{i,i}$ on a polynomial $R\in \mathcal{Q}_{n,v}(2)$. By Corollary \ref{coro Qnv} we may write
		\[ R = \sum_{\alpha}C(\alpha)\prod_{m=0}^{n}Q_{m}^{\alpha_{m}}=\sum_{\alpha}C(\alpha)Q^{\alpha},\]
		where we sum over $\alpha=(\alpha_{0},\dots,\alpha_{n})\in\Z_{\geq 0}^{n+1}$ with $|\alpha|:=\alpha_{0}+\alpha_{1}+\dots+\alpha_{n}=v$, and coefficients $C(\alpha)\in\C$. As a consequence of Lemma \ref{lemma lin indep} we can obtain relations among the coefficients $C(\alpha)$ by the vanishing of the coefficient of $\sum_{i=1}^{n}Q^{\left[\substack{i\\i}\right]}_{a}$ in $\sum_{i=1}^{n}L^{(k_{1})}_{i,i}(R)+L'^{(k_{2})}_{i,i}(R)$. Using (\ref{eq4}), (\ref{eq5}), (\ref{eq9}) and (\ref{eq10}) we find:
		
		\begin{alignat*}{3}
			&(L_{i,i}^{(k_{1})}+L'^{(k_{2})}_{i,i}&&)(R) && \\
			&= \sum_{\alpha}C(\alpha) \Bigg( &&\sum_{m=0}^{n} \alpha_{m}\frac{Q^{\alpha}}{Q_{m}} && (L_{i,i}^{(k_{1})}+L'^{(k_{2})}_{i,i})(Q_{m}) \\
			& && + \sum_{0\leq m \leq \ell\leq n} && \alpha_{m} (\alpha_{\ell}-\delta_{m,\ell})(1-\frac{\delta_{m,\ell}}{2}) \frac{Q^{\alpha}}{Q_{m}Q_{\ell}} \left((Q_{m},Q_{\ell})_{i,W}+(Q_{m},Q_{\ell})_{i,Z}\right) \Bigg) \\
			& = \sum_{\alpha}C(\alpha) \Bigg(&& \sum_{m=0}^{n} \alpha_{m}\frac{Q^{\alpha}}{Q_{m}} && \left((k_{1}+1-n+m)Q^{\left[\substack{i\\i}\right]}_{m} + (k_{2}+1-m)Q^{\left[\substack{i\\i}\right]}_{m-1} \right) \\
			& &&+ \sum_{0\leq m \leq \ell\leq n} && \alpha_{m} (\alpha_{\ell}-\delta_{m,\ell})(1-\frac{\delta_{m,\ell}}{2}) \frac{Q^{\alpha}}{Q_{m}Q_{\ell}} \bigg( 2 Q_{m}Q^{\left[\substack{i\\i}\right]}_{\ell} - 2 Q^{\left[\substack{i\\i}\right]}_{m-1} Q_{\ell+1} \\
			& && && + (Q_{m-1},Q_{\ell+1})_{i,W} + 2 Q^{\left[\substack{i\\i}\right]}_{m-1}Q_{\ell} - 2 Q_{m-1} Q^{\left[\substack{i\\i}\right]}_{\ell} + (Q_{m-1},Q_{\ell+1})_{i,Z} \bigg) \Bigg). \\
		\end{alignat*}
		For each $m=0,\dots,n$, if we sum over $i=1,2,\dots,n$, the ``coefficient'' of $\sum_{i=1}^{n}Q^{\left[\substack{i\\i}\right]}_{m}$ is some homogeneous polynomial $R_{m}$ of degree $v-1$ with respect to the polynomials $Q_{0},\dots,Q_{n}$, that is, some element of $\mathcal{Q}_{n,v-1}(2)$. If $\sum_{i=1}^{n}L^{(k_{1})}_{i,i}(R)+L'^{(k_{2})}_{i,i}(R)$ is equal to zero then by Lemma \ref{lemma lin indep}, each $R_{m}$ must be the zero polynomial. Moreover, by Lemma \ref{lemma alg ind} the scalar coefficient of each monomial (in the polynomials $Q_{0},\dots,Q_{n}$) of $R_{m}$ must in fact be zero. Therefore we extract the coeffient of the term $\frac{Q^{\alpha}}{Q_{m}}\sum_{i=1}^{n}Q^{\left[\substack{i\\i}\right]}_{m}$ for each fixed $m=0,\dots,n-1$ and $\alpha$:
		\begin{equation*}
			\begin{split} 
				0 = & C(\alpha)\alpha_{m}(k_{1}+1-n+m)  + C(\alpha-1_{m}+1_{m+1}) (\alpha-1_{m}+1_{m+1})_{m+1}(k_{2}-m) \\
				& + \sum_{\substack{0\leq \ell \leq \ell'\leq m\\ \ell+\ell'-m\geq 0}} C(\alpha(m,\ell,\ell')) \alpha(m,\ell,\ell')_{\ell}  (\alpha(m,\ell,\ell')_{\ell'}-\delta_{\ell,\ell'}) (2-\delta_{\ell,\ell'})\\
				& - \sum_{\substack{0\leq \ell\leq \ell'\leq m\\\ell+\ell'-m-1\geq 0}} C(\tilde{\alpha}(m,\ell,\ell')) \tilde{\alpha}(m,\ell,\ell')_{\ell} (\tilde{\alpha}(m,\ell,\ell')_{\ell'}-\delta_{\ell,\ell'}) (2-\delta_{\ell,\ell'})\\
				& + \sum_{\substack{m<\ell\leq \ell'\leq n\\\ell+\ell'-m-1\leq n}} C(\tilde{\alpha}(m,\ell,\ell')) \tilde{\alpha}(m,\ell,\ell')_{\ell} (\tilde{\alpha}(m,\ell,\ell')_{\ell'}-\delta_{\ell,\ell'}) (2-\delta_{\ell,\ell'})\\
				& - \sum_{\substack{m<\ell\leq\ell'\leq n\\\ell+\ell'-m\leq n}} C(\alpha(m,\ell,\ell'))\alpha(m,\ell,\ell')_{\ell} (\alpha(m,\ell,\ell')_{\ell'}-\delta_{\ell,\ell'}) (2-\delta_{\ell,\ell'}). \\
			\end{split}
		\end{equation*}
		In the relations above, we use the notation $\alpha(m,\ell,\ell'):=\alpha-1_{m}+1_{\ell}+1_{\ell'}-1_{\ell+\ell'-m}$ and $\tilde{\alpha}(m,\ell,\ell'):=\alpha-1_{m}+1_{\ell}+1_{\ell'}-1_{\ell+\ell'-m-1}$.
		We also take $C(\beta)$ to be zero if $\beta_{j}<0$ for any $0\leq j\leq n$. We fix the lexicographic order on monomials i.e. on tuples $\alpha$. That is, since all monomials concerned have the same degree $v$, 
		\[ \alpha > \beta \Leftrightarrow \exists j \text{ s.t. } \alpha_{i}=\beta_{i} \hspace*{2ex} \forall 0\leq i<j \text{ and } \alpha_{j}>\beta_{j}.  \]
		Now, fix $\alpha\neq(0,\dots,0,v)$ (so $\alpha>(0,\dots,0,v)$). Choose $m$ to be the smallest $0\leq m\leq n-1$ such that $\alpha_{m}>0$. The relation for this choice of $\alpha$ and $m$ becomes that given in the statement of Theorem \ref{thm coeff}:
		
		\begin{equation*}
			\begin{split} 
				0 = & C(\alpha)\alpha_{m}(k_{1}+1-n+m)  + C(\alpha-1_{m}+1_{m+1}) (\alpha-1_{m}+1_{m+1})_{m+1}(k_{2}-m) \\
				& + C(\alpha)\alpha_{m} (\alpha_{m}-1) \\
				& - 0\\
				& + \sum_{\substack{m<\ell\leq \ell'\leq n\\\ell+\ell'-m-1\leq n}} C(\tilde{\alpha}(m,\ell,\ell')) \tilde{\alpha}(m,\ell,\ell')_{\ell} (\tilde{\alpha}(m,\ell,\ell')_{\ell'}-\delta_{\ell,\ell'}) (2-\delta_{\ell,\ell'})\\
				& - \sum_{\substack{m<\ell\leq \ell'\leq n\\\ell+\ell'-m\leq n}} C(\alpha(m,\ell,\ell'))\alpha(m,\ell,\ell')_{\ell} (\alpha(m,\ell,\ell')_{\ell'}-\delta_{\ell,\ell'}) (2-\delta_{\ell,\ell'}). \\
			\end{split}
		\end{equation*}
		
		In particular, for each index $\alpha>(0,\dots,0,v)$ we have a linear relation determining the coefficient $C(\alpha)$ in terms of the values of coefficients $C(\beta)$ with $\beta<\alpha$. Therefore, we can write each coefficient $C(\alpha)$ in terms of $C((0,\dots,0,v))$ and so the polynomial $R$ is unique up to scaling. That is, the space of polynomials $\mathcal{H}_{n,v}(k_{1},k_{2})$ is one-dimensional and therefore by Corollary \ref{coro} the corresponding differential operator on Hermitian modular forms is unique up to scaling. \end{proof}

	\begin{remark} The coefficients of the original Rankin--Cohen brackets as given in \cite{Zag}, for example, satisfy the corresponding linear relations given in Theorem \ref{thm coeff} in the case $n=1$. Additionally, once re-written in terms of the polynomials $Q_{0}$ ,$Q_{1}$ and $Q_{2}$ the differential operators on Hermitian modular forms described by Martin--Senadheera in \cite{MarSen} have coefficients which satisfy the linear relations in Theorem \ref{thm coeff} for $n=2$, and the respective conditions on the degree $v$ agree in the case studied \textit{loc. cit.}.
		
	\end{remark}
	
	\subsection{Application to Fourier Expansions and Cusp Forms}
	As stated in \cite{eischen}, Hermitian modular forms on the unitary group $U(n,n)$ have Fourier expansions indexed by certain lattices of Hermitian matrices. That is, for a $\C$-valued Hermitian modular form $F$ on $U(n,n)$ there is some lattice $H\subset\text{Herm}_{n}(\C)$ such that $F$ has a Fourier expansion of the form
	\[ F(Z)=\sum_{h\in H}c(h)e^{2\pi i\Tr (hZ)} \]
	for some coefficients $c(h)\in\C$. Furthermore, by \cite{shimura}, Proposition 5.7, if $L^{+}\neq\Q$ or $n\neq 1$ then $c(h)=0$ unless $h$ is positive semi-definite. We assume this to be the case.
	For $\C$-valued Hermitian modular forms $F_{1},\dots,F_{r}$ of weights $k_{1},\dots,k_{r}$ respectively let the Fourier series be given by
	\[ F_{j}(Z)=\sum_{\substack{h\in H\\ h\geq 0}} c_{j}(h)e^{2\pi i\Tr (hZ)}\]
	for each $1\leq j\leq r$. For $Q\in\mathcal{H}_{n,v}(k_{1},\dots,k_{r})$ denote 
	\[G(Z):=Q\left(\frac{\partial}{\partial Z}\right)\left(F_{1}(Z_{1})\dots F_{r}(Z_{r})\right)\big|_{Z_{1}=\dots=Z_{r}=Z}. \]
	After a straightforward calculation, we can write the Fourier series for $G$ in terms of the Fourier series for $F_{1},\dots,F_{r}$ as
	\[ G(Z)=\sum_{\substack{h_{1},\dots,h_{r}\in H\\ h_{1},\dots,h_{r}\geq 0}}c_{1}(h_{1})\dots c_{r}(h_{r})Q(h_{1},\dots,h_{r})e^{2\pi i\sum_{j=1}^{r}\Tr (h_{j}Z)}. \]
	Recall that we say that a Hermitian modular form is a cusp form if the Fourier coefficients $c(h)$ are only non-zero at positive definite matrices $h$.\\
	
	\begin{proposition} \label{prop cusp} If $L^{+}\neq \Q$ or $n\neq 1$, then for any $\C$-valued Hermitian modular forms $F_{1}$ and $F_{2}$ on $U(n,n)$ of level $\Gamma$ and weights $k_{1}$ and $k_{2}$ respectively and for any $Q\in\mathcal{H}_{n,v}(k_{1},k_{2})$ with $v>0$ satisfying the hypotheses of Theorem \ref{thm}, we have that
		\[G(Z):=Q\left(\frac{\partial}{\partial Z}\right)\left(F_{1}(Z_{1}) F_{2}(Z_{2})\right)\big|_{Z_{1}=Z_{2}=Z} \]
		is a cusp form of weight $k_{1}+k_{2}+2v$ and level $\Gamma$. \end{proposition}
	
	\begin{proof} By the result of Theorem \ref{thm}, we know that $G$ is indeed a Hermitian modular form of the correct weight and level. It remains to show that $G$ is a cusp form. For each $F_{j}$ ($j=1,2$) write the Fourier expansion of $F_{j}$ as
		\[ F_{j}(Z)=\sum_{\substack{h\in H\\ h\geq 0}} c_{j}(h)e^{2\pi i\Tr (hZ)}.\]
		So, as above, the Fourier expansion of $G$ is given by
		\[ G(Z)=\sum_{\substack{h_{1},h_{2}\in H\\ h_{1},h_{2}\geq 0}}c_{1}(h_{1}) c_{2}(h_{2})Q(h_{1},h_{2})e^{2\pi i\Tr ((h_{1}+h_{2})Z)}. \]
		Since $G$ is indeed a Hermitian modular form, we know that the coefficient 
		\[c_{1}(h_{1})c_{2}(h_{2})Q(h_{1},h_{2})\]
		is equal to zero whenever $h_{1}+h_{2}$ is not positive semi-definite. It remains to show that this Fourier coefficient is also zero if $h_{1}+h_{2}$ is positive semi-definite but not positive definite. We shall show that in this case $Q(h_{1},h_{2})=0$. Indeed, if $h_{1}+h_{2}$ is not positive definite then by definition there is a vector $w$ such that
		\[ w^{\ast}(h_{1}+h_{2})w=0. \]
		Since each of $h_{1}$ and $h_{2}$ is positive semi-definite, it must be the case that $w^{\ast}h_{1}w=w^{\ast}h_{2}w=0$. Therefore
		\[ w^{\ast}(h_{1}+\lambda h_{2})w = 0 \]
		for any constant $\lambda$. In particular, $\text{det}(h_{1}+\lambda h_{2})$ is identically zero as a function of $\lambda$. With notation as in Lemma \ref{lemma Qn1}, we have 
		\[ \text{det}(h_{1}+\lambda h_{2})=\sum_{m=0}^{n}Q_{m}(h_{1},h_{2})\lambda^{m}, \]
		and so it must be that $Q_{m}(h_{1},h_{2})=0$ for each $m=0,\dots,n$. By the description of elements of $\mathcal{H}_{n,v}(k_{1},k_{2})$ in Theorem \ref{thm coeff} we know that $Q(h_{1},h_{2})$ is some homogeneous degree-$v$ polynomial in the values of $Q_{0}(h_{1},h_{2}),\dots,Q_{n}(h_{1},h_{2})$, and therefore must also be zero.
	\end{proof}

	\addcontentsline{toc}{section}{References}
	\bibliography{rcbibtex}{}
	\bibliographystyle{plain}

\end{document}